\documentclass[a4paper, 11pt]{article} 
\usepackage[margin=2.4cm, top=2.5cm, bottom=3cm, footskip=2cm]{geometry}
\usepackage{amsmath}		
\usepackage{amsfonts}		
\usepackage{amsthm}
\usepackage[utf8]{inputenc}
\usepackage{color}
\usepackage{tikz}
\usepackage{slashed}
\usetikzlibrary{calc}
\usetikzlibrary{arrows}  
\usetikzlibrary{decorations.markings}
\usetikzlibrary{shapes,snakes}
\usepackage{hyperref} 

\usepackage[figure]{hypcap}
\usepackage{amssymb}		
\usepackage{amstext}		
\usepackage{graphicx}		
\usepackage[utf8]{inputenc}	
\usepackage[T1]{fontenc}	
\usepackage{pdfpages}		
\usepackage{wrapfig}		
\usepackage{here}			
\usepackage{float}			
\usepackage{fancyhdr}
\usepackage{verbatim}		
\usepackage{wasysym}		
\usepackage{siunitx}		
\usepackage[arrow, matrix, curve]{xy} 
\usepackage{wrapfig,tabularx,multirow,hyperref,icomma}
\usepackage{mathptmx}
\usepackage{lmodern}
\usepackage{courier}
\usepackage{paralist}
\usepackage{afterpage}
\usetikzlibrary{arrows.meta}
\usepackage{subcaption}
\usetikzlibrary{decorations.pathmorphing,shapes}
\usepackage[section]{placeins}
{\left.\begin{aligned}}
	{\end{aligned}\right\rbrace}
\usepackage{titlesec}

\theoremstyle{theorem}
\newtheorem{definition}{Definition}[section]

\newtheorem{theorem}[definition]{Theorem}

\newtheorem{lemma}[definition]{Lemma}

\theoremstyle{definition}
\newtheorem{remark}[definition]{Remark}

\theoremstyle{definition}
\newtheorem*{notation}{Notation}

\theoremstyle{definition}

\newcommand{\mmysqrt}{\sqrt[\rightarrow]} 

\renewcommand{\Im}{\mathrm{Im}} 
\renewcommand{\Re}{\mathrm{Re}} 

\newcommand{\x}{\boldsymbol{x}}

\newcommand{\n}{\boldsymbol{n}}

\newcommand{\iin}{\mathrm{in}}
\newcommand{\ssc}{\mathrm{sc}}
\renewcommand{\iint}{\int\hspace{-.2cm}\int}
\newcommand{\aalpha}{\boldsymbol{\alpha}}

\renewcommand{\a}{\mathfrak{a}}
\newcommand{\GO}{\mathrm{GO}}

\usepackage{relsize}

\newcommand{\LHP}{\mathrm{LHP}}
\newcommand{\UHP}{\mathrm{UHP}}

\newcommand\numberthis{\addtocounter{equation}{1}\tag{\theequation}}
\numberwithin{equation}{section}

\graphicspath{{figures/}}

\usepackage{graphicx}
\usepackage{calc}
\usepackage{emptypage}
\usepackage{hyperref}
\usepackage{bashful}
\usepackage{natbib}

\usepackage{lastpage, hyperref}
\hypersetup{
unicode=false,
pdftoolbar=true,
pdfmenubar=true,
pdffitwindow=false,
pdfnewwindow=true,
colorlinks=true,  			
linkcolor={black!40!blue}, 			
citecolor={black!45!blue}, 
filecolor={black!45!blue}
} 

\newcommand{\contr}{\overset{\mathrm{contr.}}{\sim}}
\newcommand{\psim}{\overset{\phantom{\mathrm{contr.}}}{\sim}}

\renewcommand{\aa}{\mathrm{a}^{\star}}
\newcommand{\bb}{\mathrm{b}^{\star}}
\newcommand{\nn}{\textbf{\textit{n}}^{\star}}

\title{Diffraction by a right-angled no-contrast penetrable wedge: recovery of far-field asymptotics}  
\date{} 
\author{Valentin D. Kunz\thanks{The University of Manchester, Department of Mathematics,  Oxford Road, Manchester, M13 9PL, UK; \hfill \phantom{x}
	University of Bologna, Department of Mathematics, Piazza di Porta S. Donato, 5, 40126 Bologna, Italy} \ and  \ Raphael C. Assier\thanks{The University of Manchester, Department of Mathematics,  Oxford Road, Manchester, M13 9PL, UK} \\[.5em]
\small{Corresponding author: \href{mailto:v_kunz@startmail.com}{v$\_$kunz@startmail.com}}
}

\begin{document}

\maketitle	

\begin{abstract}
We provide a description of the far-field encountered in the diffraction problem resulting from the interaction of a monochromatic plane-wave and a right-angled no-contrast penetrable wedge. To achieve this, we employ a two-complex-variable framework and use the anayltical continuation formulae derived in (Kunz $\&$ Assier, QJMAM, 76(2), 2023) to recover the wave-field's geometrical optics components, as well as the cylindrical and lateral diffracted waves. We prove that the corresponding cylindrical and lateral diffraction coefficients can be expressed in terms of certain two-complex-variable spectral functions, evaluated at some given points.
\end{abstract}

\section{Introduction}

The present paper is a direct follow-up to \citep{KunzAssierAC}, wherein a two-complex-variable approach was employed in order to study the right-angled no-contrast penetrable wedge diffraction problem.  
This is a notoriously difficult problem, with two different wavenumbers inside and outside the wedge. So far, almost all of the methods that proved successful in the context of diffraction by a perfect wedge seem to fail in solving the penetrable wedge diffraction problem and thus, new methods need to be developed. As one of the building blocks of Keller's geometrical theory of diffraction (GTD) \citep{Keller1962, BorovikovKinber1993}, availability of such solution or far-field expansion would have a profound impact on diffraction theory overall. For such far-field expansion to be available, it is imperative to obtain analytical expressions of the corresponding diffraction coefficients, which are describing the amplitude of the diffracted far-field.

Although analytical expressions for the diffraction coefficients describing the cylindrical wave-fields emanating from the penetrable wedge's corner remain to be found, there are several methods to (approximately) compute them within some given regions. We refer to the introductions of \citep{Kunz2021diffraction, KunzAssierAC}, as well as Chapter 4 of \citep{NethercoteThesis} for a (non-exhaustive) overview of the work done on the penetrable wedge diffraction problem.

The paper \citep{KunzAssierAC} is part of an ongoing effort to apply multidimensional complex analysis to wave-diffraction problems, and other work in this direction includes \citep{AssierShanin2019, AssierShanin2021AnalyticalCont, AssierShanin2021VertexGreensFunctions, AssierAbrahams2020, AssierAbrahams2021, Kunz2021diffraction} and \citep{AssierShaninKorolkov2022}. Akin to what is done in the one-complex-variable Wiener-Hopf technique \citep{Noble1958, LawrieAbrahams2007}, whose key aspects are outlined in Fig. \ref{fig:Diagram} (left), in \citep{KunzAssierAC} it was shown that the right-angled no-contrast penetrable wedge diffraction problem admits a reformulation as a functional problem in $\mathbb{C}^2$. This functional problem involves two unknown spectral functions $\Psi_{++}$ and $\Phi_{3/4}$ in terms of which the physical fields $\psi$ and $\phi_{\ssc}$ (describing the scattered wave-fields on the interior and exterior of the wedge, respectively) can be expressed as the following inverse double Fourier-integrals. 
\begin{align}
\psi(\x)  = \frac{1}{4 \pi^2} \iint_{\mathbb{R}^2} \Psi_{++}(\aalpha) e^{- i \x \cdot \aalpha} d \aalpha \text{ and } \phi_{\ssc}(\x)  = \frac{1}{4 \pi^2} \iint_{\mathbb{R}^2} \Phi_{3/4}(\aalpha) e^{- i \x \cdot \aalpha} d \aalpha. \label{eq.Ch5phiRecovery}
\end{align}
Here, we have introduced the notation $\x=(x_1,x_2) \in \mathbb{R}^2$ and $\aalpha = (\alpha_1,\alpha_2) \in \mathbb{C}^2$, which will be used throughout the article. Moreover, following the work of \cite{AssierShanin2019}, the analyticity properties of $\Psi_{++}$ and $\Phi_{3/4}$ were studied in \citep{KunzAssierAC}, and their singularity structure within $\mathbb{C}^2$ was unveiled. 
\begin{figure}[h!]
\centering
\includegraphics[width=\textwidth]{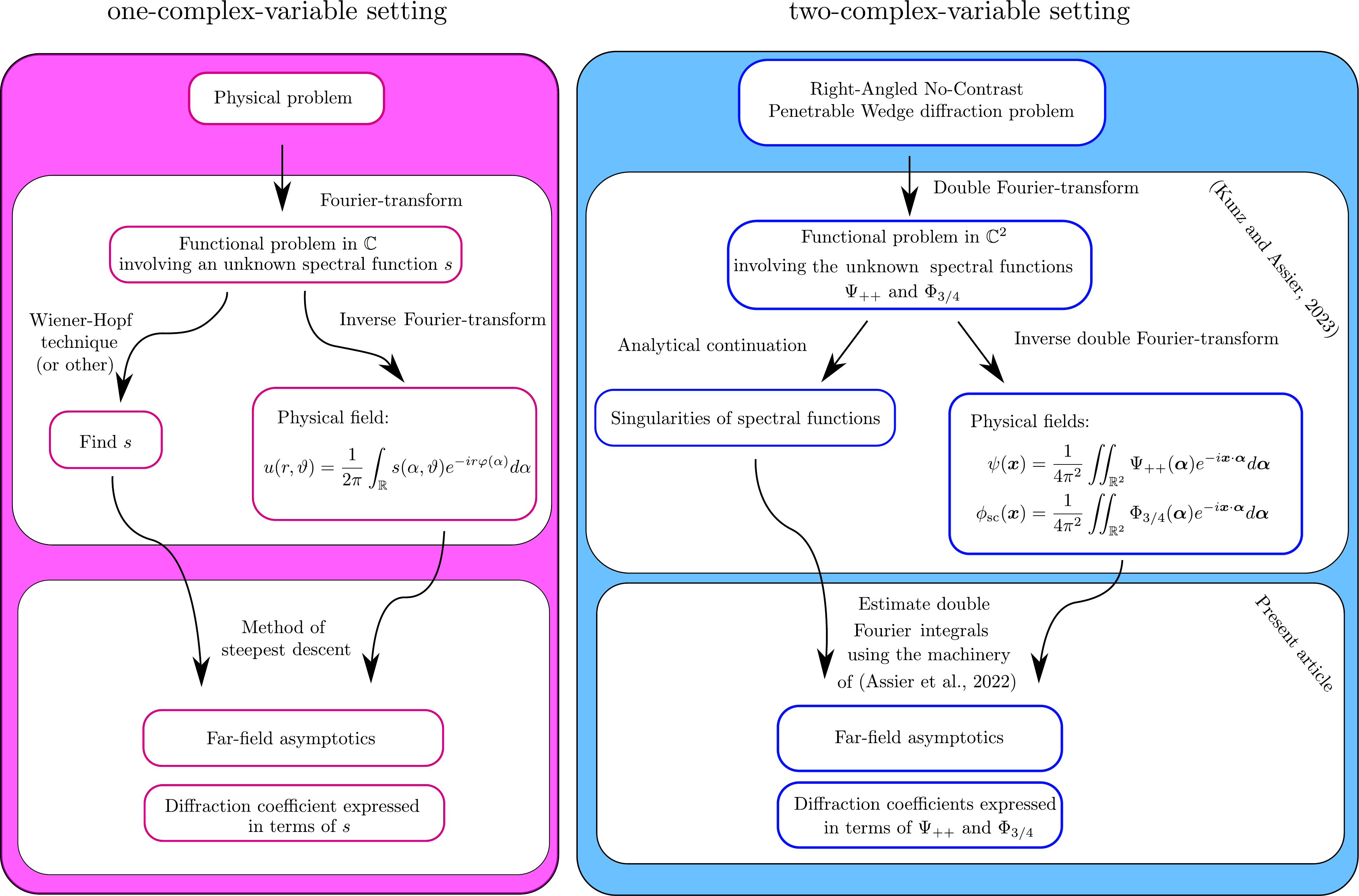}
\caption{Diagrammatic comparison of obtaining far-field asymptotics in the context of the classical, one-complex-variable, setting and the two-complex-variable setting whereon the present article is based.}
\label{fig:Diagram}
\end{figure}
Although the results obtained in \citep{KunzAssierAC} do not yet allow one to complete the two-complex-variable Wiener-Hopf technique, knowledge of the spectral functions' singularities allows for obtaining closed-form far-field asymptotics of the physical wave-fields, which are described by \eqref{eq.Ch5phiRecovery}  and obtaining such far-field asymptotics is the principal goal of the present article. The machinery required to achieve this was recently developed by Assier, Shanin, and Korolkov in \citep{AssierShaninKorolkov2022} and used to recover the far-field asymptotics of the three-dimensional quarter plane problem in \citep{AssierShaninKorolkov2023}. Particularly, in the present article we will show that, akin to the classical one-complex-variable Wiener-Hopf setting, the diffraction coefficients describing the cylindrical corner diffracted waves, as well as the diffraction coefficients describing the so-called \emph{lateral} diffracted waves, can be expressed in terms of the two-complex-variable spectral functions, evaluated at some given points. These aspects are illustrated in Fig. \ref{fig:Diagram}, right.

The content of the present article is organised as follows. After formulating the physical diffraction problem in Section \ref{sec:ProblemFormulation} and introducing some notation in Section \ref{sec:ImportantFct}, we give an informal description of its far-field in Section \ref{sec:InformalFarField} by using Keller's GTD. This article's main goal is to rigorously prove the correctness of such far-field expansion. In Section \ref{sec:Ch5Background}, the theoretical framework required to achieve this is provided, and throughout Sections \ref{sec:Simple} and \ref{sec:Complicated}, closed-form far-field asymptotics of $\psi$ and $\phi_{\ssc}$ which are in agreement with Keller's GTD are derived. We show that both the cylindrical and lateral diffraction coefficients can be expressed in terms of the spectral functions $\Psi_{++}$ and $\Phi_{3/4}$, evaluated at some given points. 
Another aim of this article is related to Sommerfeld's radiation condition. In the context of diffraction by wedges, Sommerfeld's radiation condition relies on a-priori knowledge of the physical wave-field's geometrical optics (GO) components, which, in some sense, serve as some additional boundary conditions as was pointed out by \cite{Stoker1956}. In Section \ref{sec:Complicated}, we give a formulation of the radiation which does not require a-priori knowledge of these components. 

\phantom{2}

\section{Problem formulation}\label{sec:ProblemFormulation}

Consider a plane-wave $\phi_{\iin}$ incident on an infinite, right-angled, penetrable wedge (PW) given by
\[
\text{PW} = \{(x_1,x_2) \in \mathbb{R}^2| \ x_1 \geq 0, x_2 \geq 0 \},
\]
as illustrated in Fig. \ref{fig:transparentWedge} (left).

\begin{figure}[h!]
\centering
\includegraphics[width=.8\textwidth]{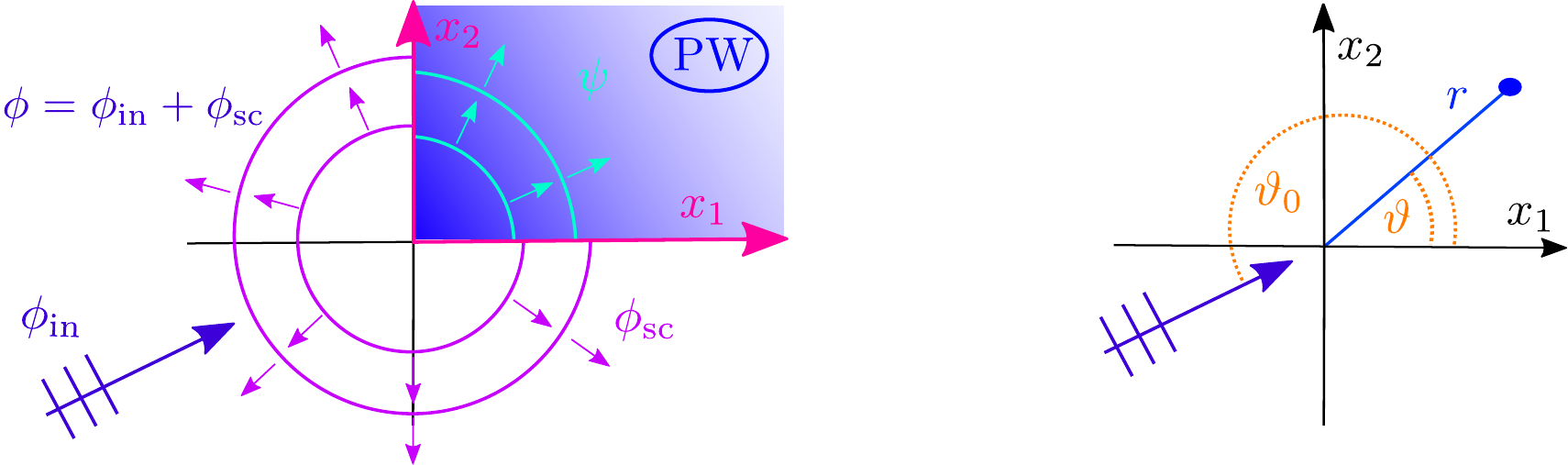}
\caption{\small Left: Illustration of the problem described by equations \eqref{eq:1.1}--\eqref{eq:1.6}, taken from \citep{KunzAssierAC}. Right: Polar coordinate system and incident angle $\vartheta_0$ of $\phi_{\iin}$.} 
\label{fig:transparentWedge}
\end{figure}

We  assume transparency of the wedge and thus expect a scattered field $\phi_{\mathrm{sc}}$ in $\mathbb{R}^2 \setminus \text{PW}$ and a transmitted field $\psi$ in $\text{PW}$. Moreover, we assume  time-harmonicity with the $e^{-i\omega t}$ convention, and the time-dependence is henceforth suppressed. Therefore, the wave-fields' dynamics are described by two Helmholtz equations, and the incident wave (only supported within $\mathbb{R}^2 \setminus \text{PW}$) is given by
\[
\phi_{\iin}(\x) = e^{i \boldsymbol{k}_1 \cdot \x}
\]
where $\boldsymbol{k}_1 \in \mathbb{R}^2$ is the incident wave vector and $\x = (x_1,x_2) \in \mathbb{R}^2$. Additionally,
we are describing a \emph{no-contrast penetrable wedge}, meaning that the \emph{contrast parameter} $\lambda$ satisfies 
\[
\lambda =1.
\]
This contrast parameter corresponds to either the ratio of magnetic permittivities or electric permeabilities, in the electromagnetic setting, or to the ratio of densities, in the acoustic setting. We refer to \citep{Kunz2021diffraction, KunzAssierAC} for a more detailed description of the physical context.

Let $k_1=|\boldsymbol{k}_1| =\omega/c_1$ and $k_2=\omega/c_2$ denote the {wavenumbers} inside and outside  PW, respectively, where $c_1$ and $c_2$ are the wave speeds relative to the media in $\mathbb{R}^2 \setminus \text{PW}$ and $\text{PW}$, respectively. Note that, although $\lambda =1$, we assume that $c_1 \neq c_2$, and so $k_1 \neq k_2$. Setting $\phi= \phi_{\mathrm{sc}} + \phi_{\iin}$ (the total wave-field in $\mathbb{R}^2 \setminus \text{PW}$), and letting $\n$ denote the inward pointing normal on $\partial \text{PW}$, the diffraction problem at hand is then described by the following equations. 

\begin{alignat}{3}
\Delta \phi + k^2_1 \phi &=0  &&\text{ in }   \mathbb{R}^2 \setminus \text{PW}, \label{eq:1.1} 	\\
\Delta \psi + k^2_2 \psi &=0  &&\text{ in }   \text{PW}, \label{eq:1.2} \\		
\phi&= \psi  &&\text{ on } \partial \text{PW}, \label{eq:1.3} \\
\partial_{\n} \phi &= \partial_{\n} \psi  &&\text{ on } \partial \text{PW}. \label{eq:1.6}	
\end{alignat}	

In the electromagnetic setting, $\phi$ and $\psi$ correspond either to the electric
or magnetic field (depending on the polarization of the incident wave, cf.\!\! \citep{Radlow1964PW,KrautLehmann1969})  in $\mathbb{R}^2 \setminus \text{PW}$ and $\text{PW}$, respectively, whereas in the acoustic setting, $\phi$ and $\psi$ represent the total pressure in  $\mathbb{R}^2 \setminus \text{PW}$ and $\text{PW}$, respectively.

Introducing polar coordinates $(r,\vartheta)$, we  rewrite the incident wave vector as $\boldsymbol{k}_1 = -k_1(\cos(\vartheta_0),\sin(\vartheta_0))$ where $\vartheta_0$ is the incident angle as shown on Fig. \ref{fig:transparentWedge}, right. We rewrite the incident wave as
\begin{align}
\phi_{\iin} = e^{-i(\mathfrak{a}_1x_1 + \mathfrak{a}_2x_2)} \label{eq:IncidentRewritten}
\end{align}
with 
\begin{align}
\mathfrak{a}_1 = k_1 \cos(\vartheta_0) \text{ and } \mathfrak{a}_2 = k_1 \sin(\vartheta_0). \label{eq.a_1,2Definition}
\end{align} 

For the problem to be well posed, and uniquely solvable, we also require the fields to satisfy the Sommerfeld radiation condition, meaning that the wave-field should be outgoing in the far-field, and edge conditions called `Meixner conditions', ensuring finiteness of the wave-field's energy near the tip  (see \citep{BabichMokeeva2008}).  The edge conditions are given by
\begin{align}
&	\phi(r, \vartheta) = B +\left(A_1\sin(\vartheta) + B_1\cos(\vartheta)\right)r + \mathcal{O}(r^2) \ \text{as} \ r \to 0, \label{eq.2.56} \\	
&	\psi(r, \vartheta) = B + \left(A_1\sin(\vartheta) + B_1 \cos(\vartheta)\right)r + \mathcal{O}(r^2) \ \text{as} \ r \to 0 \label{eq.2.57}.
\end{align}
Note that \eqref{eq.2.56} and \eqref{eq.2.57} are only valid when $\lambda =1$, and we refer to \citep{Nethercote2020HighContrastApprox} for the general case. 

To formulate the radiation condition, we henceforth consider two different cases, coined the `simple case' and the `complicated case', respectively. In these respective cases, different types of GO components are present as illustrated in Fig. \ref{fig:Wedge04}. The simple case corresponds to $\vartheta_{0} \in (\pi, 3\pi/2)$, which ensures that, for positive imaginary part $\varkappa = \Im(k_1) = \Im(k_2) >0$ of the wavenumbers, we have $\Im(\a_1) <0$ and $\Im(\a_2) <0$. In this case, the formulation of the radiation condition is straightforward: For $\varkappa >0$, the scattered and transmitted fields, $\phi_{\ssc}$ and $\psi$, decay exponentially as $r \to \infty$. In the complicated case, when $\vartheta_{0} \in (\pi/2, \pi)$ or $\vartheta_0 \in (3 \pi/2, 2\pi)$, formulating the radiation condition on the scattered and transmitted fields is not so straightforward, and will be discussed in detail in Section \ref{sec:Complicated}. 

Throughout this article, we assume that $\vartheta_{0} \neq \pi$ and that $\vartheta_0 \neq 3\pi/2$ (i.e. we do not consider the case of grazing incidence), and we assume that $\vartheta_{0} \notin (0, \pi/2)$ (i.e. the wave is incident \emph{on} the wedge). The reason for these assumptions will be explained in Sections \ref{sec:Simple} and \ref{sec:Complicated}. Moreover, due to the symmetry of the wedge, we assume without loss of generality that $\vartheta_{0} \in (\pi/2, 5 \pi/4]$. 

\section{An important function}\label{sec:ImportantFct}

\begin{figure}[h!]
\centering
\includegraphics[width=\textwidth]{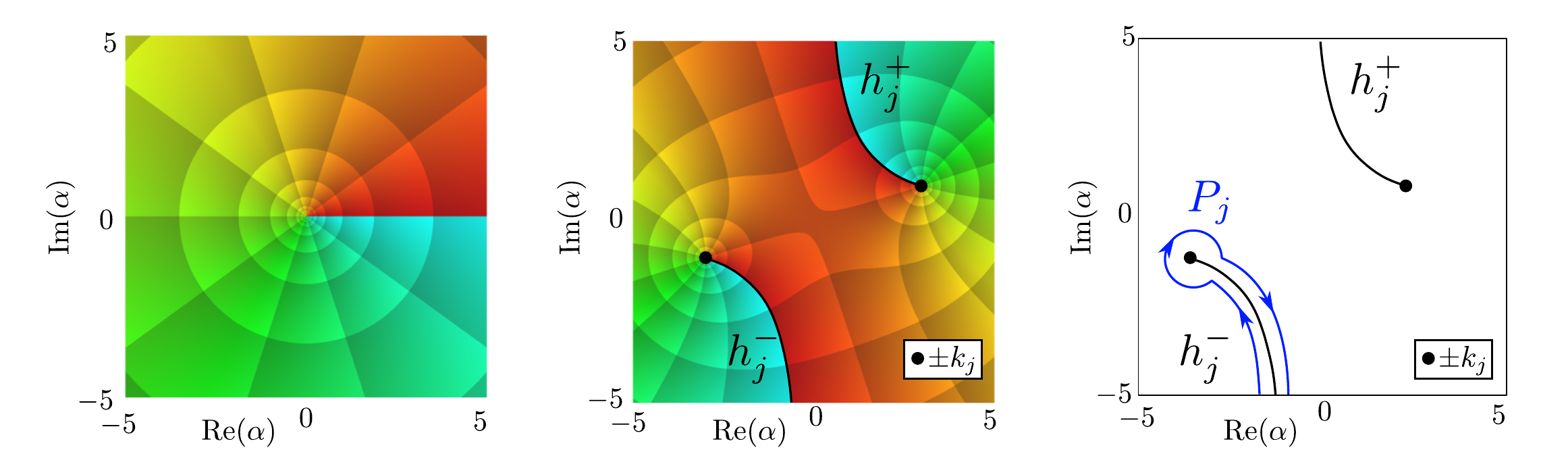}
\caption{Phase portraits of the functions $\mmysqrt{\alpha}$ (left) and $\mmysqrt{k_j^2 - \alpha^2}$ ($j=1$ or $j=2$) for $k_j=3+i$ (centre), illustration of the branch cuts $h^+$ and $h^-$ (centre and right), and of the contour $P_j$ (right).}
\label{fig:SqrtAndMysqrtKappa}
\end{figure}

Before we proceed, let us introduce the function $\mmysqrt{\alpha}$, which will be extensively used throughout the remainder of this article. We define it as the square root with branch cut on the positive real axis, and with branch determined by $\sqrt[\rightarrow]{1}=1$ (i.e.\! $\arg(\alpha) \in [0, 2\pi)$). In particular, $\Im(\sqrt[\rightarrow]{\alpha}) \geq 0$ for all $\alpha$ and $\Im(\sqrt[\rightarrow]{\alpha})=0$ if, and only if, $\alpha \in (0,\infty)$. For fixed $k_j$ (where $j=1,2$), the function
$\mmysqrt{k_j^2 - \alpha^2}$ has two branch points, at $\alpha = + k_j$ and $\alpha = - k_j$, respectively, and two corresponding branch cuts $h_j^+$ and $h_j^-$, which are given by
\begin{align}
h_j^+ = \left\{ \mmysqrt{k_j^2 - x^2}\  | \ x \in \mathbb{R} \right\} \quad \text{ and }  \quad
h_j^- = \left\{ - \mmysqrt{k_j^2 - x^2}\ | \ x \in \mathbb{R} \right\}. \label{eq.Ch1h-Def}
\end{align}
The functions $\mmysqrt{\alpha}$, $\mmysqrt{k_j^2 - \alpha^2}$, and the branch cuts $h_j^+$ and $h_j^-$ are visualised in Fig. \ref{fig:SqrtAndMysqrtKappa}. Moreover, let us define the contour $P_j$ as the (oriented) boundary of $\mathbb{C} \setminus h^-_j$, see Fig. \ref{fig:SqrtAndMysqrtKappa}. That is, for $j=1,2$, $P_j$ is the contour `starting at $-i\infty$' and moving up along $h^-_j$'s left side, up to $-k_j$, and then moving back towards $-i \infty$ along $h^-_j$'s right side. Intuitively, $P_j$ is just $h^-_j$ but `keeps track' of which side $h^-_j$ was approached from. We set $P = P_1 \cup P_2$ (this contour will be used throughout Sections \ref{sec:Ch5Background}--\ref{sec:Complicated}).

\begin{notation}
Whenever $\alpha \in (0, \infty)$, we write $\mmysqrt{\alpha} = \sqrt{\alpha}$ for simplicity. This is because for such $\alpha$, the function $\mmysqrt{\alpha}$ agrees with the usual square root function on the positive real numbers.
\end{notation}

\section{Informal description of the far-field}\label{sec:InformalFarField}\label{sec:InformalFarFieldWedge}

As $r =|\x| \to \infty$, according to Keller's GTD \citep{Keller1962}, we  expect that the wave-fields can be described by their GO components ($\phi_{\GO}$ and $\psi_{\GO}$, respectively), as well as the corresponding diffracted wave-fields resulting from the interaction of $\phi_{\iin}$ with the wedge's corner. In the far-field, the diffracted field splits into cylindrical and lateral diffracted waves, which we denote by $\phi_{\mathrm{C}}$, $\psi_{\mathrm{C}}$, and $\phi_{\mathrm{L}_1}, \ \phi_{\mathrm{L}_2}$, $\psi_{\mathrm{L}_1}, \ \psi_{\mathrm{L}_2}$, respectively. Therefore, overall, we should obtain
\begin{align}
\phi \sim \phi_{\GO} + \phi_{\mathrm{C}} + \phi_{\mathrm{L}_1} + \phi_{\mathrm{L}_2}, \text{ and } 	\psi \sim \psi_{\GO} + \psi_{\mathrm{C}} + \psi_{\mathrm{L}_1} + \psi_{\mathrm{L}_2}, \text{ as }  r \to \infty, \label{eq.Ch1PWFarfield01}
\end{align}
This article's main endeavour is to prove the correctness of \eqref{eq.Ch1PWFarfield01} and obtain formulae for the wave components $\phi_{\GO}, \ \phi_{\mathrm{C}}, \ \phi_{\mathrm{L}_1}, \ \phi_{\mathrm{L}_2}, \ \psi_{\GO}, \ \psi_{\mathrm{C}}, \ \psi_{\mathrm{L}_1},$ and $\psi_{\mathrm{L}_2}$. This will be the subject of Sections \ref{sec:Simple} (simple case) and \ref{sec:Complicated} (complicated case), respectively.

\begin{figure}[h!]
\centering
\includegraphics[width=\textwidth]{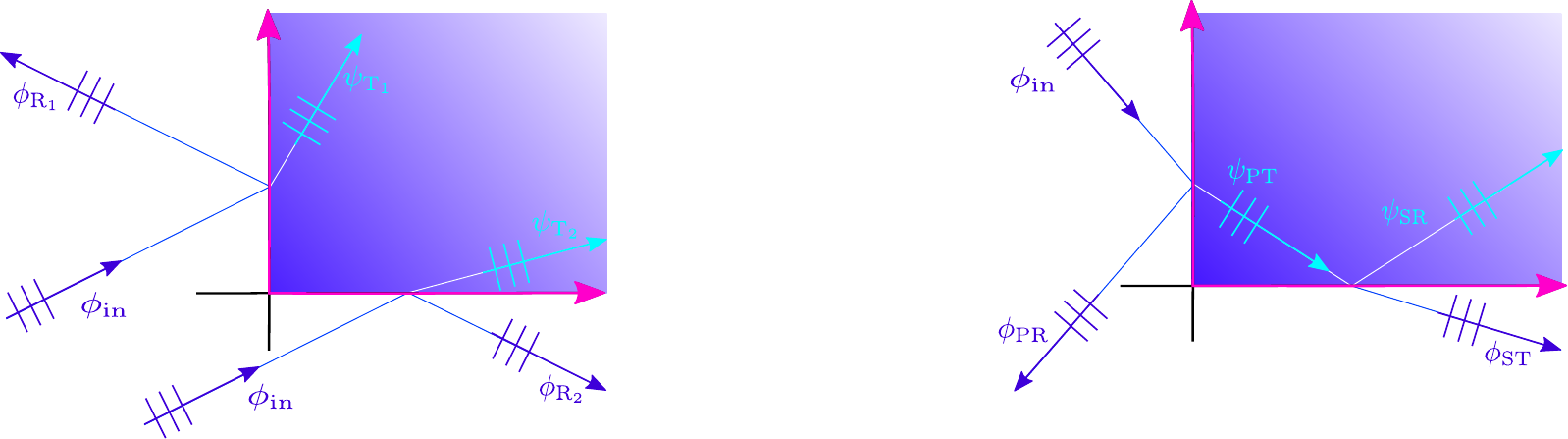}
\caption{Geometrical optics components of the wave-fields $\phi$ and $\psi$ in the simple case (left) and in the complicated case (right). The subscripts $_{\mathrm{R}_j}$, $_{\mathrm{T}_j}$, $j=1,2$ correspond to the two reflected and transmitted plane waves, whereas the subscripts $_{\mathrm{PR}}$, $_{\mathrm{PT}}$, $_{\mathrm{SR}}$, and $_{\mathrm{ST}}$ correspond to `primary reflected',  `primary transmitted', `secondary reflected', and `secondary transmitted', respectively.}
\label{fig:Wedge04}
\end{figure}

The fields' GO components are well understood and explicit formulae for $\phi_{\GO}$ and $\psi_{\GO}$  will be provided throughout Sections \ref{sec:Simple} and \ref{sec:Complicated}. The fields' GO components are displayed in Fig. \ref{fig:Wedge04} left (simple case) and right (complicated case).

The cylindrical diffracted waves $\phi_{\mathrm{C}}$ and $\psi_{\mathrm{C}}$ are a well known type of wave as well. Although no analytical description is available for them, many ways to efficiently compute them have been found. We refer to the introduction of \citep{KunzAssierAC} for an overview of the work done on computing the diffracted far-field. However, the cylindrical diffracted waves are discontinuous across the wedges interface so some additional `lateral waves' are required to ensure continuity of the total wave-fields. These diffracted lateral waves $\phi_{\mathrm{L}_1}, \ \phi_{\mathrm{L}_2}$, $\psi_{\mathrm{L}_1}$, and $\psi_{\mathrm{L}_2}$ are not as well studied. Thus, we briefly introduce this type of wave in the following section.

\subsection{Diffraction by a penetrable interface}\label{sec:PenetrableInterface}\label{sec:LateralInformal}

Consider the two half-spaces $\Omega_1 = \{(x_1,x_2) \in \mathbb{R}^2| \ x_1 >0 \}$ and $\Omega_2 = \{(x_1,x_2) \in \mathbb{R}^2| \ x_1 < 0 \}$ . We will study the diffraction problem resulting from the point source incidence given by
\begin{align}
\widetilde{\phi}_{\iin}(\x) = - \frac{i}{4} H^{(1)}_0(k_1 |\x - \x_0|), \ \x_0 \in \Omega_1.
\end{align}
Without loss of generality, we may assume that the point source is located along the $x_2$-axis, and we can thus write $\x_0 = (0, b), \ b \geq 0$, cf. Fig. \ref{fig:HalfPlaneInterface01}.

\begin{figure}[h!]
\centering
\includegraphics[width=.5\textwidth]{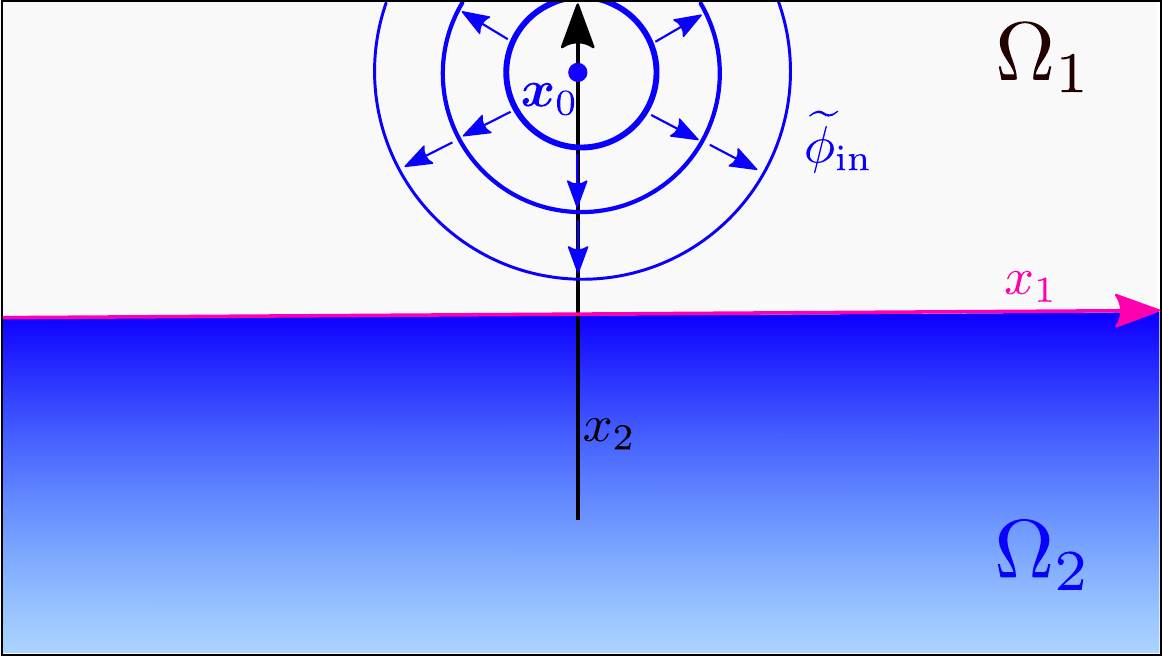}
\caption{Cylindrical wave $\widetilde{\phi}_{\iin}$ emanating from a point source at $\x_0$ in $\Omega_1$, which is incident on a half-space $\Omega_2$ whose boundary is the $x_1$-axis.}
\label{fig:HalfPlaneInterface01}
\end{figure}

Akin to the penetrable wedge diffraction problem formulated in Section \ref{sec:ProblemFormulation}, let us denote the  scattered field within $\Omega_1$ by $\widetilde{\phi}_{\ssc}$, and the total wave-field within $\Omega_2$ by $\widetilde{\psi}$. Denote the wavenumbers within $\Omega_1$ and $\Omega_2$, respectively, by $k_1$ and $k_2$. Then, the solution to the corresponding diffraction problem subject to Sommerfeld's radiation condition is given by\footnote{In the following formulae, the contour of integration is chosen to be $\mathbb{R}$. It is valid for $\Im(k_{1,2})>0$, but  for $\Im(k_1) = \Im(k_2) =0$, it needs to be an indented contour avoiding the integrands' singularities. By imposing that the scattered fields decay exponentially for positive imaginary part of the wave-numbers, finding this indentation is straightforward. We omit the details and refer to \citep{BrekhovskikGodin1999}.}
\begin{align}
&\widetilde{\phi}_{\ssc} (\x)= \frac{i}{ 4\pi} \int_{\mathbb{R}} \frac{\mmysqrt{k_1^2 - \alpha^2}  -  \mmysqrt{k_2^2 - \alpha^2}}{\mmysqrt{k_1^2 - \alpha^2}  +  \mmysqrt{k_2^2 - \alpha^2}} \frac{e^{ i (-\alpha x_1 + \mmysqrt{k_1^2 - \alpha^2}(x_2+b))}}{{\mmysqrt{k_1^2 - \alpha^2}}} d \alpha, \label{eq.HalfPlaneInterface12}\\
&\widetilde{\psi}(\x) = \frac{i}{ 4\pi} \int_{\mathbb{R}} \frac{2\mmysqrt{k_1^2 - \alpha^2}}{\mmysqrt{k_1^2 - \alpha^2}  +  \mmysqrt{k_2^2 - \alpha^2}} \frac{e^{ i (-\alpha x_1 - \mmysqrt{k_2^2 - \alpha^2} x_2 + \mmysqrt{k_1^2 - \alpha^2}b)}}{{\mmysqrt{k_1^2 - \alpha^2}}} d \alpha. \label{eq.HalfPlaneInterface13}
\end{align}
These formulae are well-known and are given, for instance, in \citep{BrekhovskikGodin1999}. We refer to \citep{KunzThesis}, Section 1.3 for their detailed derivation. \\

\noindent \textbf{Far-field asymptotics.} The far-field asymptotics of  $\widetilde{\phi}_{\ssc}$ and $\widetilde{\psi}$ can be obtained by means of the method of steepest descent. We will be brief, and refer to \citep{KunzThesis}, Section 1.3, and to the book of \cite{BrekhovskikGodin1999} for a more detailed discussion. Let us assume that $\Im(k_1) = \Im(k_2) =0$, and that $k_1 > 0$ as well as $k_2 >0$. The integrals in \eqref{eq.HalfPlaneInterface12} and \eqref{eq.HalfPlaneInterface13} have a saddle point at $\alpha = -k_1\cos(\vartheta)$ and $\alpha = -k_2\cos(\vartheta)$, respectively.  From these saddle points, we obtain a reflected and transmitted cylindrical wave which we call $\widetilde{\phi}_{\mathrm{C}}$ and $\widetilde{\psi}_{\mathrm{C}}$, respectively. However, depending on $\vartheta$ and whether $k_1 > k_2$ or $k_2 > k_1$, contributions of the branch points of the integrands need to be taken into account\footnote{The integrands' polar singularities pose no problem, as is explained in \citep{BrekhovskikGodin1999}.}.

For $k_1 > k_2$, these yield the \emph{lateral} waves within $\Omega_1$, whose leading order approximation we denote by  $\widetilde{\phi}_{\mathrm{L}_1}$ and $\widetilde{\phi}_{\mathrm{L}_2}$, and for $k_2 > k_1$ these yield the lateral waves within $\Omega_2$, with corresponding leading order approximation denoted by $\widetilde{\psi}_{\mathrm{L}_1}$ and $\widetilde{\psi}_{\mathrm{L}_2}$. For $k_1 > k_2$, they are given by 	
\begin{align}
\widetilde{\phi}_{\mathrm{L}_1}(r, \vartheta) = 
\frac{i}{2 \sqrt{\pi}}  \frac{  \sqrt{2 k_2}}{(k_1^2 - k_2^2)^{1/4}} e^{i 3 \pi/ 4}  e^{i \sqrt{k^2_1 - k_2^2}b} \frac{e^{i r ( \cos(\vartheta) k_2 + \sqrt{k_1^2 - k_2^2} \sin(\vartheta))} }{|\sqrt{k_1^2 - k_2^2}  r \cos(\vartheta) - k_2 r \sin(\vartheta)|^{3 / 2}}, \label{eq.HalfPlaneInterfaceLateral}
\end{align} 
for $\vartheta \in (0, \arccos(k_2/k_1))$ (and $\widetilde{\phi}_{\mathrm{L}_1} \equiv 0$ otherwise), and 
\begin{align}
\widetilde{\phi}_{\mathrm{L}_2}(r, \vartheta) = 
\frac{i}{2 \sqrt{\pi}}  \frac{  \sqrt{2 k_2}}{(k_1^2 - k_2^2)^{1/4}} e^{i 3 \pi/ 4} e^{i \sqrt{k^2_1 - k_2^2}b} \frac{e^{i r ( - \cos(\vartheta) k_2 + \sqrt{k_1^2 - k_2^2} \sin(\vartheta))} }{|\sqrt{k_1^2 - k_2^2}  r \cos(\vartheta) + k_2 r \sin(\vartheta)|^{3 / 2}}, \label{eq.HalfPlaneInterfaceLateral1}
\end{align}
for $\vartheta \in (\pi - \arccos(k_2/k_1), \pi)$ (and $\widetilde{\phi}_{\mathrm{L}_2} \equiv 0$ otherwise). 

For $k_2 > k_1$, they are given by
\begin{align}
\widetilde{\psi}_{\mathrm{L}_1}(r, \vartheta) = 	 \frac{i}{4\sqrt{\pi}}  \frac{ \sqrt{2 k_1} e^{i 3 \pi/ 4}}{(k_2^2 - k_1^2)^{1/4}}   \frac{e^{i r ( \cos(\vartheta) k_1 - \sqrt{k_2^2 - k_1^2} \sin(\vartheta))} }{|\sqrt{k_2^2 - k_1^2}  r \cos(\vartheta) + k_1 r \sin(\vartheta)|^{3 / 2}}, \label{eq.HalfPlaneInterfaceLateral02} 
\end{align}
and  
\begin{align}
\widetilde{\psi}_{\mathrm{L}_2}(r, \vartheta) = 	 \frac{i}{4\sqrt{\pi}}  \frac{ \sqrt{2 k_1} e^{i 3 \pi/ 4}}{(k_2^2 - k_1^2)^{1/4}}   \frac{e^{i r (-\cos(\vartheta) k_1 - \sqrt{k_2^2 - k_1^2} \sin(\vartheta))} }{|-\sqrt{k_2^2 - k_1^2}  r \cos(\vartheta) + k_1 r \sin(\vartheta)|^{3 / 2}}.
\end{align}
Although $\widetilde{\psi}_{\mathrm{L}_1}$ and $\widetilde{\psi}_{\mathrm{L}_2}$ are independent of $b$, the transmitted field's lateral waves do, in fact, depend on $b$, but this dependence only appears for higher order correction terms.
Whenever $\vartheta \notin (-\arccos(k_1/k_2), 0)$, we have $\widetilde{\psi}_{\mathrm{L}_1} \equiv 0$, and whenever $\vartheta \notin (- \pi, -\pi + \arccos(k_1/k_2))$, we have $\widetilde{\psi}_{\mathrm{L}_2} \equiv 0$. For $j=1,2$, let us set $\widetilde{\phi}_{\mathrm{L}_j} \equiv 0$ whenever $k_1 \leq k_2$ and $\widetilde{\psi}_{\mathrm{L}_j} \equiv 0$ whenever $k_2 \leq k_1$.
Overall, as $r \to \infty$, we then obtain
\begin{align}
\widetilde{\phi}_{\ssc} \sim
\widetilde{\phi}_{\mathrm{C}} + \widetilde{\phi}_{\mathrm{L}_1} + \widetilde{\phi}_{\mathrm{L}_2} \text{ and }
\widetilde{\psi} \sim 
\widetilde{\psi}_{\mathrm{C}} + \widetilde{\psi}_{\mathrm{L}_1} + \widetilde{\psi}_{\mathrm{L}_2},
\end{align}	
where
\begin{align}
\widetilde{\phi}_{\mathrm{C}}(r,\vartheta) =  -\frac{i}{4} \frac{\sqrt{k_1^2 - k_1^2 \cos^2(\vartheta)}  -  \sqrt{k_2^2 - k_1^2 \cos^2(\vartheta)}}{\sqrt{k_1^2 - k_1^2 \cos^2(\vartheta)}  +  \sqrt{k_2^2 - k_1^2 \cos^2(\vartheta)}} 
e^{k_1\sin(\vartheta)b}H^{(1)}_0(k_1 r), \label{eq.HalfPlaneInterface14}
\end{align}	
and
\begin{align}
\widetilde{\psi}_{\mathrm{C}}(r, \vartheta) =   - \frac{i}{2}\frac{ \sqrt{k_1^2 - k_2^2 \cos^2(\vartheta)}}{\sqrt{k_1^2 - k_2^2 \cos^2(\vartheta)} +  \sqrt{k_2^2 - k_2^2 \cos^2(\vartheta)}} e^{i \sqrt{k_1^2 - k_2^2\cos^2(\vartheta)} b} H^{(1)}_0(k_2 r).
\end{align}
The far-field asymptotics of $\widetilde{\phi}$ and $\widetilde{\psi}$ are shown in Fig. \ref{fig:Ch3HeatMap}.
\begin{figure}
\centering
\includegraphics[width=\textwidth]{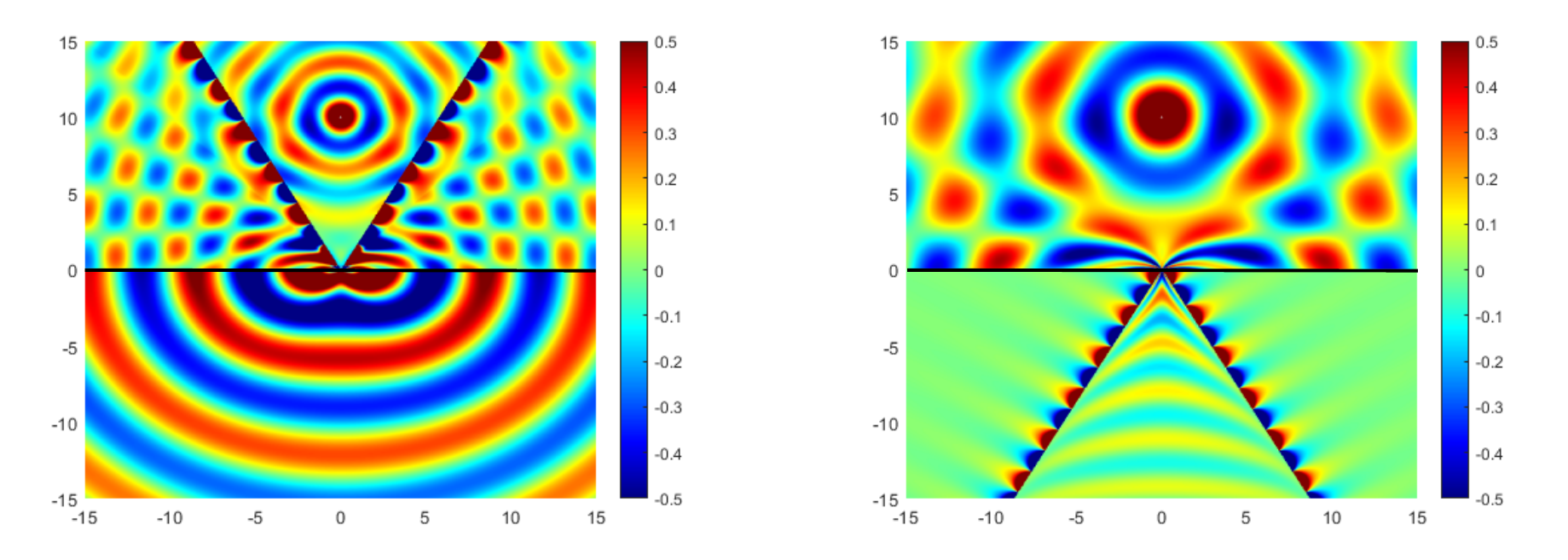}
\caption{Plot of the wave-fields' far-field asymptotics for $k_1=2$ and $k_2=1$ (left), and $k_1=1$ and $k_2=2$ (right). Here, we have $b=10$.}
\label{fig:Ch3HeatMap}
\end{figure}

\begin{notation}
Henceforth, we shall also refer to the lateral waves' leading order approximation as `the' lateral wave. That is, we refer to $\widetilde{\phi}_{\mathrm{L}_j}$ and $\widetilde{\psi}_{\mathrm{L}_j}, \ j=1,2$ as lateral waves, and similarly for the penetrable wedge diffraction problem. This is consistent with the notation used in \eqref{eq.Ch1PWFarfield01}. Since we will only be interested in such leading order approximations, this notation is unambiguous.
\end{notation}

Note that the lateral waves $\widetilde{\phi}_{\mathrm{L}_1}$ and $\widetilde{\phi}_{\mathrm{L}_2}$ decay away from the lines $\vartheta \equiv \arccos(k_2/k_1)$ and $\vartheta \equiv \pi-\arccos(k_2/k_1)$, respectively, whereas the lateral waves $\widetilde{\psi}_{\mathrm{L}_1}$ and $\widetilde{\psi}_{\mathrm{L}_2}$ decay away from the lines $\vartheta \equiv - \arccos(k_1/k_2)$ and $\vartheta \equiv -\pi+\arccos(k_1/k_2)$, respectively. On these lines, the lateral waves are undefined, which is due to the coalescence of the respective branch points and the saddle point of the integrands in equations \eqref{eq.HalfPlaneInterface12} and \eqref{eq.HalfPlaneInterface13}. So, on these lines, the diffracted wave does not split into cylindrical and lateral waves. Such coalescence also occurs for $\vartheta=0$ or $\vartheta = \pi$, so, on the interface, the diffracted wave does not split as well. Finally, note that since the lateral waves are of order $\mathcal{O}(r^{-3/2})$, they decay faster than the corresponding cylindrical waves, which are of order $\mathcal{O}(r^{-1/2})$. 
For more on lateral waves, their physical meaning, and their relevance, we refer to \citep{BrekhovskikGodin1999}, Chapter 3. 

\section{Background and setup}\label{sec:Ch5Background}

To prove the correctness of \eqref{eq.Ch1PWFarfield01}, we will rely on using the frameworks recently developed in \citep{KunzAssierAC} and \citep{AssierShaninKorolkov2022}. Thus, we summarise these papers' key results in Sections \ref{sec:Ch5OSummary01} and \ref{sec:Ch5OSummary02}, respectively.

\subsection{Summary of \citep{KunzAssierAC}}\label{sec:Ch5OSummary01}
As before, let 
$$\varkappa = \Im(k_{1,2})\geq 0$$ 
denote the imaginary part of $k_1$ and $k_2$. In the following formulae, we omit the argument $\varkappa$ for brevity. Let $\aalpha = (\alpha_1, \alpha_2) \in \mathbb{C}^2$, set $\psi \equiv 0$ within $\mathbb{R}^2 \setminus \text{PW}$ and $\phi_{\ssc} \equiv 0$ within $\text{PW}$, and introduce the (unknown) spectral functions $\Psi_{++}$ and $\Phi_{3/4}$ via
\begin{align*}
\Psi_{++}(\aalpha) = \iint_{\mathbb{R}^2} \psi(\x) e^{i \aalpha \cdot \x} d\x, \text{ and } \Phi_{3/4}(\aalpha) = \iint_{\mathbb{R}^2} \phi_{\ssc}(\x) e^{i \aalpha \cdot \x} d\x.
\end{align*}
Then,
according to Theorem 4.2.1 of \citep{KunzAssierAC}, these spectral functions satisfy the following two-complex-variable Wiener-Hopf functional equation
\begin{align}
- K(\aalpha) \Psi_{++}(\aalpha) = \Phi_{3/4}(\aalpha) + P_{++}(\aalpha), \label{eq.Ch5WHeq}
\end{align}
where the kernel $K$ and the forcing $P_{++}$ are given by 
\begin{align}
K(\aalpha) = \frac{k_2^2 - \alpha_1^2 - \alpha_2^2}{k_1^2 - \alpha_1^2 - \alpha_2^2} \quad \text{ and } \quad P_{++}(\aalpha) = \frac{1}{(\a_1 - \alpha_1)(\a_2 - \alpha_2)}\cdot \label{eq.Ch5KPDef}
\end{align}
The Wiener-Hopf equation is valid in the (common) domain of analyticity of $K$, $P_{++}$, $\Psi_{++}$, and $\Phi_{3/4}$, which includes some open neighbourhood $U \subset \mathbb{C}^2$ of $\mathbb{R}^2$. We proceed to describe this analyticity structure. To this end, let us introduce the following sets.
\begin{align*}
\UHP = \{ \alpha \in \mathbb{C}| \ \Im(\alpha) \geq 0 \}, \ \LHP = \{ \alpha \in \mathbb{C}| \ \Im(\alpha) \leq 0 \}, \ 
H^- = \LHP \setminus \left( h^-_1 \cup h^-_2 \right).
\end{align*}
henceforth, when speaking of analyticity of a function on a closed set, it is implied that there is some open neighbourhood of the set's boundary whereon this function is analytic. 
Let
\begin{align}
\mathcal{I}_1(\aalpha; z_1) = \frac{\left(k^2_2 - k_1^2 \right) \Psi_{++}\left(z_1, \mmysqrt{k^2_1 -z^2_1}\right)}{K_{-\circ }(z_1,\alpha_2)(z_1-\alpha_1)\left(\mmysqrt{k^2_1 -z^2_1} - \alpha_2\right)\mmysqrt{k^2_1 -z^2_1}}, \label{eq.I1Def} \\
\mathcal{I}_2(\aalpha; z_2) = \frac{\left(k^2_2 - k_1^2 \right) \Psi_{++}\left(\mmysqrt{k^2_1 -z^2_2}, z_2\right)}{K_{\circ -}(\alpha_1,z_2)(z_2-\alpha_2) \left(\mmysqrt{k^2_1 -z^2_2} - \alpha_1\right)\mmysqrt{k^2_1 -z^2_2}}. \label{eq.I2Def}
\end{align}
Then, in the simple case, for $\varkappa >0$, the following formulae for analytic continuation can be used to prove that the function $\Psi_{++}(\aalpha)$ is analytic within the domain 
\begin{align*}
\left(\UHP \times \mathbb{C} \setminus (h_1^- \cup h_2^- \cup \{\a_2\})  \right) \cup \left( \mathbb{C} \setminus (h_1^- \cup h_2^- \cup \{\a_1\}) \times \UHP  \right).
\end{align*}
\begin{align*}	
\Psi_{++}(\aalpha)= \frac{-i}{4 \pi K_{\circ+}(\aalpha)} \int_{P} \mathcal{I}_2(\aalpha; z_2) dz_2  
-
\frac{K_{-\circ}(\alpha_1,\mathfrak{a}_2)P_{++}(\aalpha)}{K_{\circ+}(\aalpha) K_{\circ-}(\alpha_1,\mathfrak{a}_2)K_{-\circ}(\mathfrak{a}_1,\mathfrak{a}_2)},
\numberthis  \label{eq.Ch5PsiCont}  \\
\Psi_{++}(\aalpha)= \frac{-i }{4 \pi K_{+\circ}(\aalpha)}  \int_{P}  \mathcal{I}_1(\aalpha; z_1) dz_1 
- 
\frac{K_{\circ-}(\mathfrak{a}_1,\alpha_2) P_{++}(\aalpha)}{K_{+\circ}(\aalpha) K_{-\circ}(\mathfrak{a}_1,\alpha_2)K_{\circ-}(\mathfrak{a}_1,\mathfrak{a}_2) } \label{eq.Ch5PsiCont2}, \numberthis 
\end{align*}
where the contour $P$ was defined in Section \ref{sec:ImportantFct}. Formula \eqref{eq.Ch5PsiCont} is valid within the domain $H^-\setminus \{\a_1\} \times \UHP$ whereas formula \eqref{eq.Ch5PsiCont2} is valid within $\UHP \times H^-\setminus \{\a_2\}.$
Upon multiplication by $K$, formulae \eqref{eq.Ch5PsiCont} and \eqref{eq.Ch5PsiCont2} can be used to analytically continue $\Phi_{3/4}$ onto the domain $(H^-\setminus \{\a_1\}) \times (H^- \setminus \{\a_2\})$. 
Recall that we have $
\a_1 = k_1 \cos(\vartheta_0) \text{ and } \a_2 = k_1 \sin(\vartheta_0),  $
where $\vartheta_0$ is the incident angle (cf.\! Section \ref{sec:ProblemFormulation}), and 
\begin{align}
K_{- \circ}(\aalpha) = \frac{\mmysqrt{k_2^2 - \alpha_2^2} - \alpha_1}{\mmysqrt{k_1^2 - \alpha_2^2} - \alpha_1}, \
K_{+ \circ}(\aalpha) = \frac{\mmysqrt{k_2^2 - \alpha_2^2} + \alpha_1}{\mmysqrt{k_1^2 - \alpha_2^2} + \alpha_1}, \label{eq.Kalpha1Fact} \\
K_{\circ - }(\aalpha) = \frac{\mmysqrt{k_2^2 - \alpha_1^2} - \alpha_2}{\mmysqrt{k_1^2 - \alpha_1^2} - \alpha_2}, \
K_{\circ + }(\aalpha) = \frac{\mmysqrt{k_2^2 - \alpha_1^2} + \alpha_2}{\mmysqrt{k_1^2 - \alpha_1^2} + \alpha_2}, \label{eq.Kalpha2Fact}
\end{align}
where the function $\mmysqrt{\alpha}$ is as defined in Section \ref{sec:ImportantFct}, and we have $K=K_{\circ +}K_{\circ -} = K_{+ \circ} K_{- \circ}$.

In \citep{KunzAssierAC}, it was proved that from the spectral functions, the physical fields $\psi$ and $\phi_{\sc}$ can be obtained via \eqref{eq.Ch5phiRecovery}.

As already mentioned, formulae \eqref{eq.Ch5PsiCont}--\eqref{eq.Ch5PsiCont2} have been derived for the simple case $\vartheta_0 \in (\pi, 3 \pi/2)$ and for $\varkappa >0$ only. However, these formulae remain valid in the limit $\varkappa \to 0$, provided that the integration contours are appropriately changed to avoid the singularities of $\Psi_{++}(z_1, \mmysqrt{k_1^2 - z^2_1})$ or $\Psi_{++}(z_1, \mmysqrt{k_1^2 - z^2_1})$ which may hit $P$ in this limit. Now, to recover the physical wave-fields' far-field asymptotics and thus prove the correctness of \eqref{eq.Ch1PWFarfield01}, we need to make sense of  \eqref{eq.Ch5phiRecovery} in the limit $\varkappa \to 0$. Taking this limit is non-trivial, however, and it requires to change the surface $\mathbb{R}^2$ to another surface $\boldsymbol{\Gamma}$, by using the generalised theorem of Stokes, such that $\boldsymbol{\Gamma}$ does not hit any of the spectral functions' singularities during this limit. We refer to \citep{Shabat1991} and \citep{MadsenTornehave1997} for more on Stokes' theorem.  The singularity structure of the spectral functions was thoroughly studied in \citep{KunzAssierAC}, and will be provided in Sections \ref{sec:Simple}, for the simple case, and in Section \ref{sec:Complicated}, for the complicated case.


\subsection{Obtaining far-field asymptotics}\label{sec:Ch5OSummary02}

Since we will rely on using the machinery developed in \citep{AssierShaninKorolkov2022}  in order to recover the physical far-field asymptotics of the scattered and transmitted fields associated to the right-angled no-contrast penetrable wedge diffraction problem, let us briefly summarise this paper's key results.

Consider an integral of the form 
\begin{align}
f(\x; \varkappa) = \frac{1}{4 \pi^2}\iint_{\mathbb{R}^2} F(\aalpha; \varkappa) e^{- i \x \cdot \aalpha} d \aalpha, \label{eq.Ch5ModelIntegral}
\end{align}
where $\varkappa \geq 0$ is some (small) real, positive parameter.
For example, $f = \psi$ or $f= \phi_{\ssc}$ where $\varkappa = \Im(k_2)$ or $\varkappa = \Im(k_1)$, respectively. Let us assume that the singularity set $\sigma$ of $F$ consists (only) of poles and branches. We moreover assume that 
$$\sigma=\cup_j \sigma_j$$
is the union of so-called \emph{irreducible singularities} $\sigma_j$. That is, $\sigma_j = \{\aalpha \in \mathbb{C}^2| \ g_j(\aalpha) =0 \}$ is the zero set of some holomorphic function $g_j(\aalpha)$ such that $(\partial_{\alpha_1}g_j,\partial_{\alpha_2}g_j) \neq 0$ on $\sigma_j$. Such $g_j$ is henceforth called a \emph{defining function} of $\sigma_j$. The preceding assumptions imply that our irreducible singularities are \emph{regular analytic sets}. For a detailed introduction to the latter, we refer to \citep{Shabat1991} and \citep{Chirka1989}.

The function $F(\aalpha; \varkappa)$ is chosen such that, when $\varkappa >0$, we have $\sigma \cap \mathbb{R}^2 = \emptyset$. However, when $\varkappa \to 0$, $\sigma$ hits the real plane at its real trace $\sigma'$, which is defined by 

$$\sigma' = \sigma \cap \mathbb{R}^2.$$ 
If it is possible to deform the surface $\mathbb{R}^2$ continuously to another surface $\boldsymbol{\Gamma}$ such that $\sigma \cap \boldsymbol{\Gamma} = \emptyset$ for all sufficiently small $\varkappa$, then, by Stokes' theorem, we can define $f(\x; 0)$ as 
\begin{align}
f(\x; 0 ) = \frac{1}{4 \pi^2}\iint_{\boldsymbol{\Gamma}} F(\aalpha; 0) e^{- i \x \cdot \aalpha} d \aalpha. \label{eq.Ch5ModelIntegral02}
\end{align}
Due to Stokes' theorem, again, this integral is completely determined by the relative position of $\boldsymbol{\Gamma}$ to $\sigma$ (as long as no singularities are hit, $\boldsymbol{\Gamma}$ can be deformed onto any other surface $\boldsymbol{\Gamma}'$). This relative position can be described by the bridge and arrow notation, which was first introduced in \citep{AssierShanin2021VertexGreensFunctions}, and further developed in \citep{AssierShaninKorolkov2022}. Below, we briefly outline its key aspects.

\subsubsection{The bridge and arrow notation}\label{sec:Ch5BridgeArrow}

We only consider functions with the following \emph{real property}. 
\begin{definition}\label{def:aabbnn}
Let $\sigma = \cup \sigma_j$ where the $\sigma_j$ are irreducible singularities with respective defining functions $g_j$. We define
\begin{align}
	\aa_j = \partial_{\alpha_1} g_j(\aalpha^{\star}), \ \bb_j =\partial_{\alpha_2} g_j(\aalpha^{\star}),
\end{align}
and say that  $\sigma_j$ has the real property when for every $\aalpha^{\star} \in \mathbb{R}^2$, we have $g_j(\aalpha^{\star}) \in \mathbb{R}$, $\aa_j \in \mathbb{R}$, and $\bb_j \in \mathbb{R}$. Since $\sigma_j$ is regular, we have $(\aa_j)^2 + (\bb_j)^2 \neq 0$. In this case, the real trace $\sigma'=\sigma \cap \mathbb{R}^2$ is a smooth, one dimensional curve.
\end{definition}

Let us now discuss how the deformation $\mathbb{R}^2 \to \boldsymbol{\Gamma}$ can be achieved practically. The following discussion is informal, however, and we refer to \citep{AssierShaninKorolkov2022} for a more rigorous formulation of the concepts introduced below.

Since the singularity $\sigma_j$ depends, generally, on the parameter $\varkappa$, let us for now write $\sigma_j= \sigma_j(\varkappa)$, and similarly $g_j(\aalpha)=g_j(\aalpha;\varkappa)$. However, since $\sigma_j' = \sigma_j(0) \cap \mathbb{R}^2$,  it does not depend on $\varkappa$, and we will just write $\sigma_j'$ without ambiguity.

Let $\aalpha^{\star} \in \sigma_j'$. Let us analyse how the integration surface $\mathbb{R}^2$ needs to be changed for the integral \eqref{eq.Ch5ModelIntegral} to be well-defined in the limit $\varkappa \to 0$. Without loss of generality, we may assume that $\aa_j \neq 0$ (if $\aa_j =0$, we must have $\bb_j \neq 0$ and the following procedure can be repeated). 
Take the complex plane $\{\alpha_2 \equiv \alpha_2^{\star}\} = \{\aalpha \in \mathbb{C}^2| \ \alpha_2 = \alpha_2^{\star}\}$. By the implicit function theorem, there exists a function $\alpha_1^{\dagger}(\varkappa)$ such that the equation 
\begin{align}
g_j(\alpha_1^{\dagger}(\varkappa), \alpha_2^{\star}; \varkappa) =0
\end{align}
is solvable for all sufficiently small $\varkappa$ and therefore $\sigma_j(\varkappa) \cap \{\alpha_2 \equiv \alpha_2^{\star}\} = \{(\alpha_1^{\dagger}(\varkappa), \alpha_2^{\star})\}.$ Clearly, $\{\alpha_2 \equiv \alpha_2^{\star}\} \cap \mathbb{R}^2 = \{\aalpha \in \mathbb{R}^2| \alpha_2 = \alpha_2^{\star} \}$ and therefore, the projection of this intersection onto the complex $\alpha_1$-plane is simply the real $\alpha_1$-line. Since for $\varkappa >0$, we have $\sigma_j(\varkappa) \cap \mathbb{R}^2 = \emptyset$, the point $\alpha_1^{\dagger}(\varkappa)$ must have non-zero imaginary part. Therefore, $\alpha_1^{\dagger}(\varkappa)$ lies either above, or below, the real $\alpha_1$ line, as illustrated in Fig. \ref{fig:Ch5Bypass01} (left). 

Now, it is possible to continuously change the surface $\mathbb{R}^2$ to another surface $\boldsymbol{\Gamma}$ such that, during this change, the singularity $\sigma_j(\varkappa)$ is never hit, for any $\varkappa$, and such that $\boldsymbol{\Gamma} \cap \sigma_j' = \emptyset$. This is shown in Fig. \ref{fig:Ch5Bypass01} (centre), where the curve $\gamma$ is the projection of $\boldsymbol{\Gamma}$ onto the complex $\alpha_1$-plane. For our purpose, it is enough to consider surfaces $\Gamma$ that can be parametrised over the real plane: $\Gamma$ can be described by $\{\aalpha_{\boldsymbol{\Gamma}}(\aalpha^r) \in \mathbb{C}^2| \ \aalpha^r = (\alpha_1^r,\alpha_2^r) \in \mathbb{R}^2\}$, where 
\begin{align*}
\aalpha_{\boldsymbol{\Gamma}}(\aalpha^r) = (\alpha_1^r + i\eta_1(\aalpha^r), \alpha_2^r + i \eta_2(\aalpha^r))
\end{align*}
for some vector field $\boldsymbol{\eta}: \mathbb{R}^2 \to \mathbb{R}^2$. For $\boldsymbol{\Gamma}$ to be suitable, we need to choose $\boldsymbol{\eta}$ such that $\boldsymbol{\eta}(\aalpha^{\star})$, for $\aalpha^{\star} \in \sigma_j', \ \boldsymbol{\eta}(\aalpha^{\star})$ is not tangent to $\sigma_j'$ and not zero.  For a given $\aalpha^{\star} \in \sigma_j'$, we call $\boldsymbol{\eta}(\aalpha^{\star})$ the \emph{arrow} at the point $\aalpha^{\star}$. Since $\boldsymbol{\eta}$ is never tangent to $\sigma_j'$, nowhere on $\sigma_j'$ vanishing, and since $\sigma_j'$ is a curve in $\mathbb{R}^2$, the arrow must point either to the right, or to the left of $\sigma_j'$, as illustrated in Fig. \ref{fig:Ch5Bypass01}, right\footnote{Here, the notion of `left' and `right' is as follows: right, if for $\aa_1 \neq 0$ $\boldsymbol{\eta}$ is on the same side as the positive $\Re(\alpha_1)$ direction, and left otherwise.}. It can be shown that if $\alpha_1^{\dagger}(\varkappa)$ lies above the real $\alpha_1$-axis, then the arrow points to the left of $\sigma_j'$ and, similarly, if  $\alpha_1^{\dagger}(\varkappa)$ lies below the real $\alpha_1$-axis, then the arrow points to the right of $\sigma_j'$ (see Fig. \ref{fig:Ch5Bypass01}). We refer to \citep{AssierShaninKorolkov2022}, Section 3 for the proof of all of the above statements.

We can therefore visualise the relative position of $\boldsymbol{\Gamma}$ to $\sigma_j$ by using the bridge and arrow symbol, which we `attach' to the singularity's real trace $\sigma_j'$, as shown in Fig. \ref{fig:Ch5Bypass01} (right).

\begin{figure}[t!]
\centering
\includegraphics[width=\textwidth]{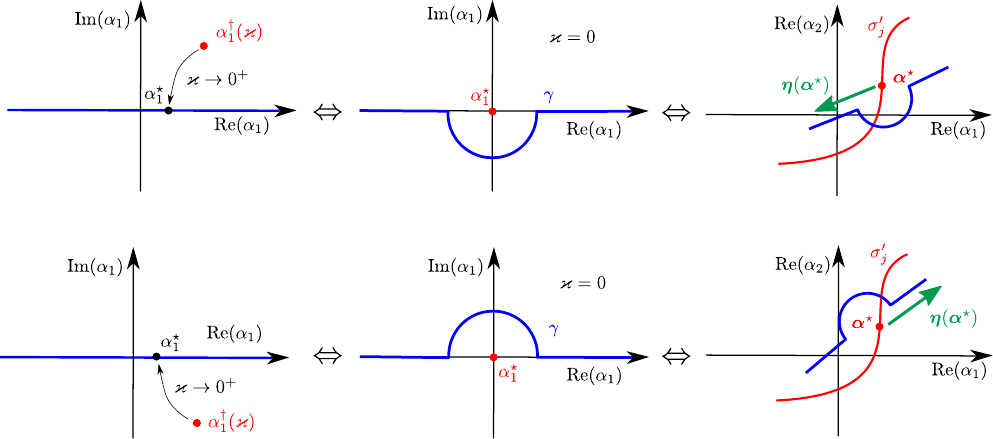}
\caption{The only two possible types of indentation of the deformed surface $\boldsymbol{\Gamma}$ about $\sigma_j$, and the resulting bridge and arrow configuration.}
\label{fig:Ch5Bypass01}
\end{figure}

\begin{notation}
Henceforth, we shall omit the argument $\varkappa$ whenever $\varkappa =0$.
\end{notation}

The bridge and arrow configuration has the crucial, and practically very useful, property that it can be continuously carried along a singularity's real-trace, as illustrated in Fig. \ref{fig:Ch5Bypass04}. That is, if the arrow points to a given side of $\sigma_j'$ at one point $\aalpha^{\star} \in \sigma_j'$, then it must point to this side at each point $\aalpha \in \sigma_j'$.
Moreover, we have the following \emph{tangential touch compatibility} (which is also illustrated in Fig.  \ref{fig:Ch5Bypass04}). If two singularities $\sigma_1$ and $\sigma_2$ intersect tangentially at $\aalpha^{\star}_j \in \sigma'_1 \cap \sigma'_2$, then the arrow points to the same side of both, $\sigma'_1$ and $\sigma'_2$.
\begin{figure}[h!]
\centering
\includegraphics[width=\textwidth]{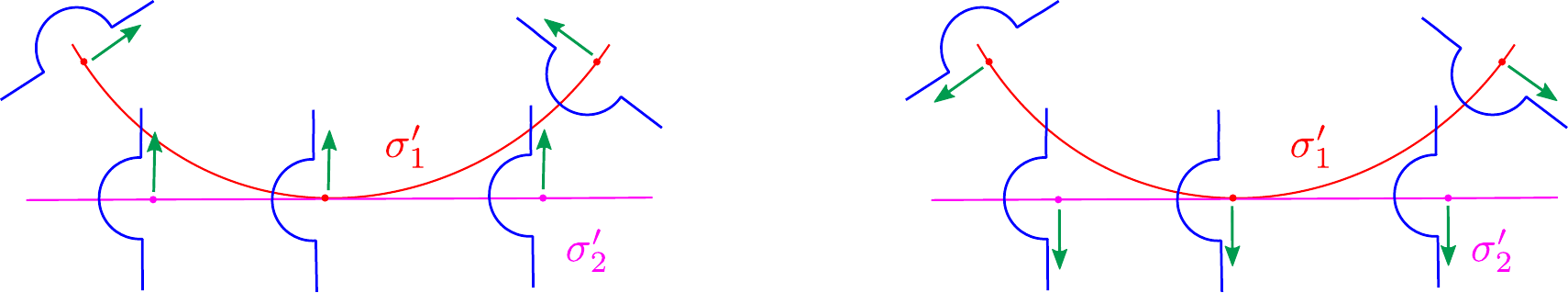}
\caption{The two possible bridge and arrow configurations for a tangential crossing.}
\label{fig:Ch5Bypass04}
\end{figure}

\subsubsection{The locality principle}\label{sec:Ch5locality}
Crucial for estimating integrals of the form \eqref{eq.Ch5ModelIntegral} is the following \emph{locality principle}. This principle relies on the concept of \emph{contributing points}, which we  explain after the theorem is stated.

\begin{theorem}[\citep{AssierShaninKorolkov2022} Theorem 6.1]\label{thm:Ch5Locality}
The integral $f(\x)$ can be estimated as $|\x| \to \infty$ as an asymptotic series. For almost all observation directions $\tilde{\x} = \x/|\x|$, the terms of the asymptotic series (up to terms decaying exponentially as $|\x|\to \infty$) can be obtained by estimating the integral in the neighbourhoods of some real, isolated, contributing points.
\end{theorem}

\paragraph{Contributing points} 

The only potentially contributing points, in the sense of Theorem	\ref{thm:Ch5Locality}, are the so-called `saddles on singularities' and transversal crossings of singularities. We explain the concept of a saddle on a singularity (SOS) below, but first, let us examine what this implies regarding the estimation of $f$. If the point $\aalpha^{\star} \in \sigma_j'$ is not a crossing of singularities and not an SOS, there exists an open neighbourhood $U^{\star} \subset \mathbb{C}^2$ of $\aalpha^{\star}$ (which intersects $\boldsymbol{\Gamma}$) such that the integral
\begin{align}
f_{\mathrm{loc}}(\x) = \iint_{\boldsymbol{\Gamma} \cap U^{\star}} F(\aalpha) e^{- i \x \cdot \aalpha} d\aalpha \label{eq.Ch5Estimate01} 
\end{align}
decays exponentially as $|\x| \to \infty$\footnote{To show this, the integration surface $\boldsymbol{\Gamma}$ may need to be deformed once more to another surface, $\boldsymbol{\Gamma}'$, say, on which $\tilde{\x} \cdot \boldsymbol{\eta}' < 0$. Again, we refer to \citep{AssierShaninKorolkov2022} for the technical details involved.}.
Similarly, $f_{\mathrm{loc}}(\x)$ decays exponentially if $\aalpha^{\star}$ is a point of tangential crossings of singularities, or if $\aalpha^{\star}$ is not singular. Although $f_{\mathrm{loc}}(\x)$ depends on the point $\aalpha^{\star}$, it will always be clear which point $\aalpha^{\star}$ is referred to, such that this notation is non-ambiguous.

\begin{definition}[Contributing point]
If $f_{\mathrm{loc}}(\x)$ does not decay exponentially, we call $\aalpha^{\star}$ a contributing point. 
\end{definition}

\begin{definition}[SOS]
Fix an observation direction $\tilde{\x}$ defined by $\x = r \tilde{\x}$ (i.e.\! $\tilde{\x} = (\cos(\vartheta), \sin(\vartheta))$). Let $\tilde{\x}$ be orthogonal to $\sigma_j'$ at the point $\aalpha^{\star} \in \sigma_j'$. If the bridge and arrow configuration  is such that $\tilde{\x} \cdot \boldsymbol{\eta}(\aalpha^{\star}) >0$, then $\aalpha^{\star}$ is called an SOS on $\sigma_j$ with respect to $\tilde{\x}$. 
\end{definition}

\paragraph{Estimation of integrals}\label{sec:Ch5IntegralEstimation} 
We now outline how the locality principle can be used in practice.
For an irreducible singularity $\sigma_j$, let 
\begin{align*}
\ \nn_j = \frac{1}{\sqrt{(\aa_j)^2 + (\bb_j)^2}}
\begin{pmatrix}
	\aa_j \\
	\bb_j
\end{pmatrix}.
\end{align*}

\noindent \textbf{Transversal crossing of singularities.} Let $\aalpha^{\star} \in \sigma_1' \cap \sigma_2'$ and let us assume that the singularities $\sigma_1$ and $\sigma_2$ cross transversally at $\aalpha^{\star}$.
Set $\Delta^{\star} = \aa_1 \bb_2 - \aa_2 \bb_1$, and let us assume that $\Delta^{\star} >0$\footnote{Since the crossing is transversal we always have $\Delta^{\star} \neq 0$, and by eventually changing the sign of the defining function $g_1$, say, the condition $\Delta^{\star} >0$ is non-restrictive.}. Introduce the sign factors $s_1$ and $s_2$ as follows. For $j=1,2$, we set $s_j=+1$ (respectively $s_j=-1$), if  $\text{sign}(\nn_j \cdot \boldsymbol{\eta}(\aalpha^{\star})) > 0$ (resp. $\text{sign}(\nn_j \cdot \boldsymbol{\eta}(\aalpha^{\star})) < 0$). Let $\tilde{\x}$ be such that $\aalpha^{\star}$ is not an SOS with respect to $\tilde{\x}$.  If there exist constants $A \in \mathbb{C}$ and $m_{1,2} \in \mathbb{R}$ such that
\begin{align}
F(\aalpha) \sim A \times g_1^{-m_1}(\aalpha) \times g_2^{-m_2}(\aalpha), \text{ as } \aalpha \to \aalpha^{\star}
\end{align}
then, as $r \to \infty$, we have
\begin{align}
f_{\mathrm{loc}}(\x) \sim \frac{A e^{-i \x \cdot \aalpha^{\star}} e^{-i \frac{\pi}{2}\left(s_1 m_1+s_2 m_2\right)}}{\Gamma\left(m_1\right) \Gamma\left(m_2\right)\left(\Delta^{\star}\right)^{m_1+m_2-1}}\frac{\mathcal{H}(s_1(x_1 \bb_2 - x_2 \aa_2))}{|x_1 \bb_2-x_2 \aa_2 |^{1-m_1}}\frac{\mathcal{H}(s_2(-x_1\bb_1 + x_2 \aa_1))}{|-x_1 \bb_1+x_2 \aa_1 |^{1-m_2}},  \label{eq.Ch5EstimateCrossing}
\end{align}
where $\mathcal{H}$ is the Heaviside step-function, and $\Gamma$ is the gamma function (see \citep{AssierShaninKorolkov2022}, Section 5.2). If $m_1$ or $m_2$ is a negative integer, no crossing of singularities occurs, and $f_{\text{loc}}$ is exponentially decaying. \\ 

\noindent \textbf{Isolated SOS.}  We will only consider\emph{ isolated} SOS: in some small open neighbourhood $U^{\star}$ of $\aalpha^{\star}$, the point  $\aalpha^{\star}$ is the only SOS on $\sigma'_j$ with respect to $\tilde{\x}$, and $\aalpha^{\star}$ does not belong to any other singularity. Similar to the case of a transversal crossing discussed above, we introduce the sign factor $s$ as follows. $s=+1$, if $\text{sign}(\nn_j \cdot \boldsymbol{\eta}(\aalpha^{\star})) > 0$, and $s=-1$, if $\text{sign}(\nn_j \cdot \boldsymbol{\eta}(\aalpha^{\star})) < 0$. Let now $\aalpha^{\star}$ be an isolated SOS with respect to $\tilde{\x}$. If there exist constants $A \in \mathbb{C}$ and $m_{} \in \mathbb{R}$ such that
\begin{align}
F(\aalpha) \sim A \times g_j^{- m_{}}(\aalpha), \text{ as } \aalpha \to \aalpha^{\star}
\end{align}
then, as $r \to \infty$, we have
\begin{align}
f_{\mathrm{loc}}(\x) \sim \frac{ A e^{-i \x \cdot \aalpha^{\star}} \sqrt{\pi} e^{-i s m_{} \pi / 2}}{2 \pi \left(\left(\mathrm{a}_j^{\star}\right)^2+\left(\mathrm{b}_j^{\star}\right)^2\right) \Gamma(m_{})}\left(\frac{r}{\sqrt{\left(\mathrm{a}_j^{\star}\right)^2+\left(\mathrm{b}_j^{\star}\right)^2}}\right)^{m_{}-3 / 2}  \times	
\begin{cases}
	\begin{aligned}
		\frac{e^{-i \pi / 4}}{\sqrt{s \xi}} &\text { if } s \xi>0, \\
		\frac{e^{i \pi / 4}}{\sqrt{-s \xi}} &\text { if } s \xi<0,
	\end{aligned}
\end{cases} \label{eq.Ch5EstimateSOS}
\end{align}
see \citep{AssierShaninKorolkov2022}, Section 5.1 (as in the case of a transversal crossing of singularities, if $m$ is a negative integer, $f_{\text{loc}}$ is exponentially decaying).
Here, the constant $\xi$ is defined as follows. Set 
\begin{align}
\Lambda = \aa(\alpha_1 - \alpha_1^{\star}) + \bb (\alpha_2 - \alpha_2^{\star}), \text{ and } 
\zeta =  \bb (\alpha_1 - \alpha_1^{\star}) - \aa (\alpha_2 - \alpha_2^{\star}).
\end{align}
Then, it can be shown that there exists a unique constant $\xi \in \mathbb{R}$ such that
\begin{align}
g(\aalpha) = \Lambda - \xi \zeta^2 + \mathcal{O}\left(\Lambda^2 + \zeta \Lambda\right) \text{ as } \aalpha \to \aalpha^{\star}.
\end{align}
The proof can be found Appendix A of \citep{AssierShaninKorolkov2022}, and shows that $\xi$ is proportional to the curvature of $\sigma_j'$ at $\aalpha^{\star}$.

Thus, overall, to obtain the sought asymptotic expansion of $f$, it is enough to estimate $F$ near the potentially contributing points.

\begin{definition}[Wave component]
If $f_{\mathrm{loc}}(\x)$ is not exponentially decaying, we call the wave-field associated with the leading order far-field asymptotics of $f_{\mathrm{loc}}(\x)$, as described by \eqref{eq.Ch5EstimateCrossing} and \eqref{eq.Ch5EstimateSOS}, a  wave component of $f$. 
\end{definition}

Therefore, to prove the correctness of \eqref{eq.Ch1PWFarfield01}, we need to show that the  wave components of $\phi$ are given by $\phi_{\GO}$, $\phi_{\mathrm{C}}$, $\phi_{\mathrm{L}_1}$, and $\phi_{\mathrm{L}_2}$ and, similarly, that the  wave components of $\psi$ are given by $\psi_{\GO}$, $\psi_{\mathrm{C}}$, $\psi_{\mathrm{L}_1}$, and $\psi_{\mathrm{L}_2}$.

Before we proceed with the corresponding calculation, let us  outline the general strategy that we will follow throughout the rest of this article. Let $\aalpha^{\star}$ be either a transversal crossing of singularities or an isolated SOS.
\begin{enumerate}
\item[\textbf{Step 0.}] \  Determine the bridge and arrow configuration at $\aalpha^{\star}$.
\item[\textbf{Step 1.}] \ Choose defining functions $g_1$ and $g_2$ such that $\Delta^{\star} >0$ (transverse crossing), or some defining function $g$ (SOS). This allows for determining the sign factors and the constant $\xi$.
\item[\textbf{Step 2.}] \ Study the asymptotic behaviour of $F= \Psi_{++}$ or $F= \Phi_{3/4}$ near $\aalpha^{\star}$, and determine whether $f_{\mathrm{loc}}(\x)$ is exponentially decaying.
\item[\textbf{Step 3.}] \ If $f_{\mathrm{loc}}(\x)$ does not decay exponentially, use the formulae presented in Section \ref{sec:Ch5IntegralEstimation} to obtain the corresponding  wave component.
\end{enumerate}
Moreover, the following notation will be used throughout the remainder of this article.
\begin{notation}[Contributing asymptotic behaviour]
Let $F$ be any function for which the asymptotic behaviour as $\aalpha \to \aalpha^{\star}$ is to be found. Let us assume that, as $\aalpha \to \aalpha^{\star}$, we have 
\begin{align*}
	F(\aalpha) \sim F_1(\aalpha) + F_2(\aalpha),
\end{align*}
and that, depending on whether $\aalpha^{\star}$ is an SOS or a point of transversal crossings of singularities, $F_1(\aalpha)$ has no isolated SOS on $\aalpha^{\star}$ (relative to the observation direction $\tilde{\x}$ of interest), or $F_1(\aalpha)$ does not exhibit a transversal crossing of singularities at $\aalpha^{\star}$, respectively. We then write
\begin{align*}
	F(\aalpha) \contr F_2(\aalpha), \ \text{ as } \aalpha \to \aalpha^{\star}.
\end{align*}
This notation is motivated by the fact, that we are interested in the asymptotic behaviour of $$f_{\mathrm{loc}}(\x) = \frac{1}{4 \pi^2} \iint_{\boldsymbol{\Gamma} \cap U^{\star}} F(\aalpha) e^{- i \x \cdot \aalpha} d \aalpha,$$
and the integral over $F_1$ must decay exponentially. Since we are only interested in the  wave components associated with $f_{\mathrm{loc}}$, the integral over $F_1$ may thus be discarded. In other words, the \emph{contributing} asymptotic behaviour of $F$ is completely determined by the asymptotic behaviour of $F_2$.
\end{notation}

\section{The simple case}\label{sec:Simple}

Recall that the `simple case' corresponds to $\vartheta_{0} \in (\pi, 3 \pi/2)$ where $\vartheta_{0}$ is the incident angle.  In \citep{KunzAssierAC}, it was shown that we then have the following irreducible singularities of $\Psi_{++}$ and $\Phi_{3/4}$. 
\begin{align*}
\sigma_{p_1} = \{\aalpha \in \mathbb{C}^2| \ \alpha_1 = \a_1 \}, \ 
\sigma_{p_2} = \{\aalpha \in \mathbb{C}^2| \ \alpha_1 = \a_2 \}, \ 
\sigma_{b_1} = \{ \aalpha \in \mathbb{C}^2| \ \alpha_1 = -k_1 \} , \\
\sigma_{b_2} = \{ \aalpha \in \mathbb{C}^2| \ \alpha_1 = -k_2 \}, \ 
\sigma_{b_3} = \{ \aalpha \in \mathbb{C}^2| \ \alpha_2 = -k_1 \}, \ 
\sigma_{b_4} = \{ \aalpha \in \mathbb{C}^2| \ \alpha_2 = -k_2 \}.   
\end{align*}
The subscript $_p$ indicates a polar singularity, whereas the subscript $_b$ denotes a branch. Here, the assumption that $\vartheta_0 \neq \pi, 3 \pi/2$ is important, since it ensures that the polar and branch sets are not coalescing.	Moreover, the set $\sigma_{c_1}$, introduced below, is an irreducible polar singularity of $\Phi_{3/4}$ whereas the set $\sigma_{c_2}$ is an irreducible polar singularity of $\Psi_{++}$.
\begin{align*}
\sigma_{c_1} = \{\aalpha \in \mathbb{C}^2| \ \alpha_1^2 + \alpha_2^2 = k_1^2\}, \quad
\sigma_{c_2} = \{\aalpha \in \mathbb{C}^2| \ \alpha_1^2 + \alpha_2^2 = k_2^2\}.
\end{align*}
Observe that all of these irreducible singularities are regular and have the real property. Their real traces are shown in Fig. \ref{fig:Ch5Singularities01}. Thus, the framework outlined in Section  \ref{sec:Ch5OSummary02} is indeed applicable. In \citep{KunzAssierAC}, it was shown that only parts of $\sigma_{c_1}$ and $\sigma_{c_2}$ are actually singular. This is why we only show parts of the circles in Fig. \ref{fig:Ch5Singularities01}.
Now, using the procedure described in Section 3 of \citep{AssierShaninKorolkov2022} (which is summarised in Section \ref{sec:Ch5BridgeArrow}), it is reasonably straightforward to find the bridge and arrow configuration on the sets $\sigma'_{p_j}, j=1,2$ and $\sigma'_{b_j}, j=1,2,3,4$. Then, by tangential-touch compatibility, the bridge and arrow configuration can be found on $\sigma'_{c_1}$ and $\sigma'_{c_2}$. The resulting  configurations are shown in Fig. \ref{fig:Ch5Singularities01}.
\begin{figure}[h]
\includegraphics[width=\textwidth]{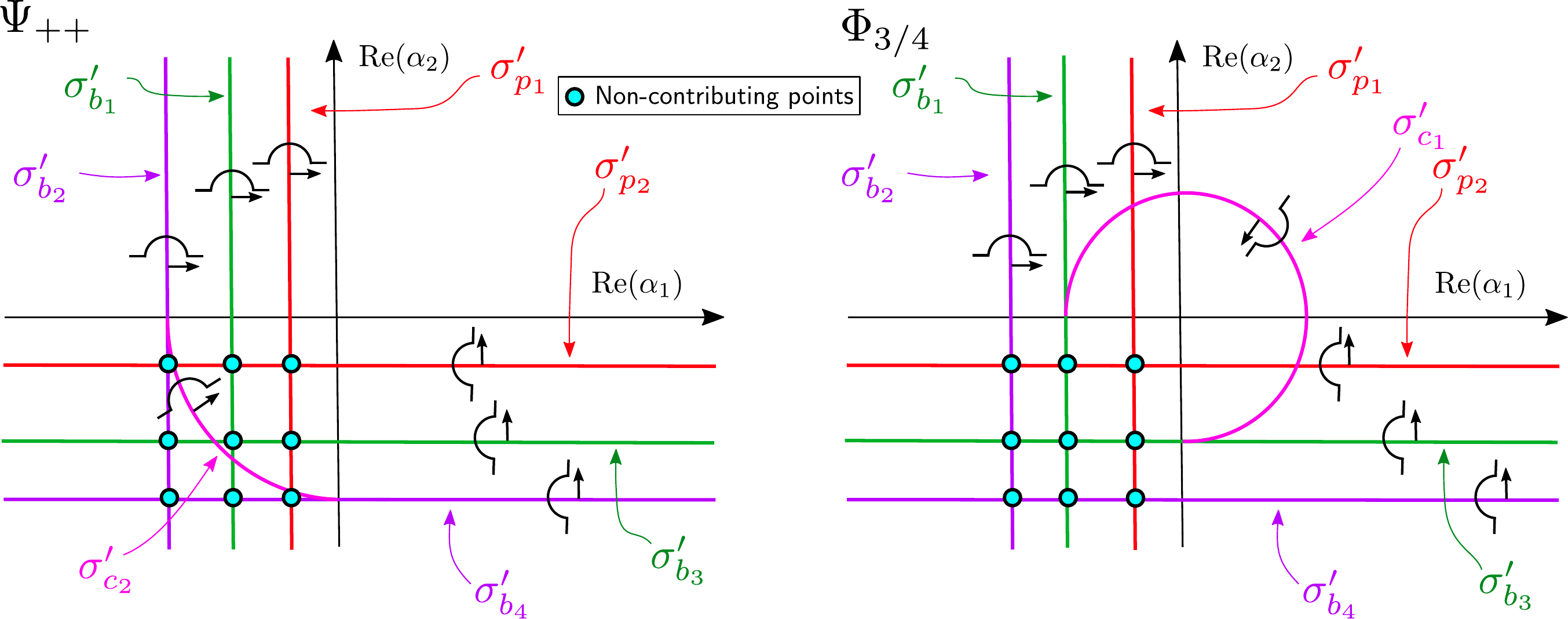}
\caption{Real traces of the spectral functions' irreducible singularities, corresponding bridge and arrow configuration, and non-contributing points.}
\label{fig:Ch5Singularities01}\label{fig:Ch5Singularities02}
\label{fig:Ch5SimpleTransmitted}\label{fig:PsiCW}	\label{fig:Ch5LateralTransmitted}\label{fig:SaddleSC} \label{fig:Ch5LateralReflected}
\end{figure}

Henceforth, let $\x = r \tilde{\x}$ with $\tilde{\x} = (\cos(\vartheta), \sin(\vartheta))$, and let us assume that $\vartheta \neq 0, \pi/2$. Recall that the only points that can potentially yield  wave components are transversal crossings of singularities and SOS. Since $\vartheta \neq 0, \pi/2$, only points on $\sigma_{c_1}'$ and $\sigma_{c_2}'$, respectively, can be such an SOS. We proceed by studying these potentially contributing points. Recall that, we only consider $k_1 \neq k_2$ for, otherwise, the penetrable wedge diffraction problem described in Section \ref{sec:ProblemFormulation} is trivial.

\subsection{Non-contributing crossings}\label{sec:Ch5Additive01}

In this section, we look at all crossings that do not involve the circles $\sigma'_{c_1}$ and $\sigma'_{c_2}$. That is, we analyse the crossings $\sigma'_{p_1} \cap \sigma'_{p_2}$, $\sigma'_{p_j} \cap \sigma'_{b_l}, \ j=1,2; \ l=1,2,3,4$ and $\sigma'_{b_j} \cap \sigma'_{b_l}, \ j,l = 1,2,3,4$. In Section 4 of \citep{KunzAssierAC}, it was shown that the crossings $\sigma'_{b_j} \cap \sigma'_{b_l}, \ j,l = 1,2,3,4$ are `additive' in the sense, that $\Phi_{3/4}$ satisfies the additive crossing property relative to these crossing points (we refer to \citep{AssierShanin2019}, \citep{AssierShaninKorolkov2022}, or \citep{KunzAssierAC} for an explanation of this property). Using the Wiener-Hopf equation \eqref{eq.Ch5WHeq}, it can be shown that the spectral function $\Psi_{++}$ also satisfies the additive crossing property relative to these crossings. By Theorem 4.13 of \citep{AssierShaninKorolkov2022}, such crossings therefore do not yield a  wave component. Now, it can be shown that, in fact, all of the crossings highlighted in Fig. \ref{fig:Ch5Singularities02} are non-contributing, which implies that they do not yield a  wave component. The procedure for proving this is the same for each crossing, however, and we therefore only give the proof for the point $\aalpha^{\star} = (-k_1, \a_2)$, corresponding to the crossing of $\sigma'_{b_1}$ and $\sigma'_{p_2}$.

\begin{lemma}\label{lem:01}
The point $\aalpha^{\star} = (-k_1, \a_2)$ is non-contributing.		
\end{lemma}	

\begin{proof} We follow the strategy outlined at the end of Section \ref{sec:Ch5IntegralEstimation}, and show that $\aalpha^{\star}$ is non-contributing for $\Phi = -\Phi_{3/4} - P_{++}$. This implies the sought non-contributing behaviour of $\Psi_{++}$ (by the Wiener-Hopf equation) and  of $\Phi_{3/4}$ (since $P_{++}$ does not exhibit a crossing of singularities at $\aalpha^{\star}$). Let $\sigma_1 = \sigma_{b_1}$ and $\sigma_2 = \sigma_{p_2}$. 

\textbf{Step 1.} Choose $g_1(\aalpha)= \alpha_1 + k_1$ and $g_2(\aalpha) = \alpha_2 - \a_2$ as defining functions for the singularities $\sigma_1$ and $\sigma_2$, respectively. 

\textbf{Step 2.} For $\varkappa >0$, we have $(-k_1, \a_2)  \in \LHP \times \LHP$, so we can use equation \eqref{eq.Ch5PsiCont} or \eqref{eq.Ch5PsiCont2} for computing the asymptotic behaviour of $\Phi$ as $\aalpha \to \aalpha^{\star}$. 
Let us use  formula \eqref{eq.Ch5PsiCont}, wherein only the external additive term exhibits a crossing of singularities. Therefore, as $\aalpha \to \aalpha^{\star}$, we find 
\begin{align*} 
	-\Phi(\aalpha) & \contr  
	\frac{K_{\circ -}(\aalpha) K_{- \circ}(-k_1, \a_2)}{K_{\circ - }(\alpha_1, \a_2) K_{- \circ}(\a_1, \a_2) (-k_1 - \a_1)(\alpha_2 - \a_2)}\cdot \label{eq.NonContr1} \numberthis
\end{align*}
Now, using the definition of $K_{\circ -}$ (cf.\! \eqref{eq.Kalpha1Fact}), we have, as $\aalpha \to \aalpha^{\star}$
\begin{align}
	\frac{K_{\circ -}(\aalpha)}{K_{\circ -}(\alpha_1,\a_2)} = \frac{\mmysqrt{k_2^2 - \alpha_1^2} - \alpha_2}{\mmysqrt{k_1^2 - \alpha_1^2} - \alpha_2} \frac{\mmysqrt{k_1^2 - \alpha_1^2} - \a_2}{\mmysqrt{k_2^2 - \alpha_1^2} - \a_2} \sim 	\frac{\mmysqrt{k_1^2 - \alpha_1^2} - \a_2}{\mmysqrt{k_1^2 - \alpha_1^2} - \alpha_2},
\end{align}
and, moreover, we have
\begin{align*}
	\frac{\mmysqrt{k_1^2 - \alpha_1^2} - \a_2}{\mmysqrt{k_1^2 - \alpha_1^2} - \alpha_2} 
	&= \frac{k_1^2 - \alpha_1^2  - \a_2 \alpha_2}{k_1^2 - \alpha_1^2 - \alpha_2^2} + \frac{\mmysqrt{k_1^2 - \alpha_1^2}(\alpha_2 - \a_2)}{k_1^2 - \alpha_1^2 - \alpha_2^2}\cdot \numberthis
\end{align*}
Therefore, as $\aalpha \to \aalpha^{\star}$, we find 
\begin{align*}
	-\Phi(\aalpha)  \contr & \underbrace{\frac{ K_{- \circ}(-k_1, \a_2)}{ K_{- \circ}(\a_1, \a_2) (-k_1 - \a_1)(\alpha_2 - \a_2)}}_{U_1(\aalpha)}  + \underbrace{\frac{ K_{- \circ}(-k_1, \a_2) \mmysqrt{k_1^2 - \alpha_1^2} }{ K_{- \circ}(\a_1, \a_2) (-k_1 - \a_1)(- \a_2^2)}}_{U_2(\aalpha)}\cdot  \numberthis \label{eq.Ch5Additive01}
\end{align*}
Now, the first term, $U_1(\aalpha)$, is regular at $\alpha_1 = -k_1$ whereas the second term, $U_2(\aalpha)$, is regular at $\alpha_2 = \a_2$. Therefore, neither $U_1$ nor $U_2$ exhibits a crossing of singularities at $\aalpha = \aalpha^{\star}$ and thus, the crossing is additive. Therefore, by Theorem 4.13 of \citep{AssierShaninKorolkov2022}, it does not lead to a  wave component.
\end{proof}

\begin{remark}
Analysing the crossing $(\a_1,\a_2)$ of $\sigma_{p_1}$ and $\sigma_{p_2}$ is slightly different, but easier: proceeding as in the proof of Lemma \ref{lem:01}, we find that $\Phi \contr P_{++}$. This means that the crossing $(\a_1,\a_2)$ is not additive for $\Phi$. However we get $\Phi_{3/4} \contr 0$, so the crossing does not yield a wave component.
\end{remark}

\subsection{Wave components of $\psi$}

Let us now study the remaining transversal crossings as well as the points which are isolated SOS of $\Psi_{++}$. Recall that $\tilde{\x} = (\cos(\vartheta), \sin(\vartheta))$. When $k_2 < k_1$, the remaining potentially contributing points of $\Psi_{++}$ are given by 
\begin{align*}
\aalpha_{\mathrm{T}_1} = \left(- \sqrt{k_2^2 - \a_2^2}, \a_2 \right), \ \aalpha_{\mathrm{T}_2} = \left(\a_1, - \sqrt{k_2^2 - \a_1^2}\right), \ \aalpha_{\mathrm{C}_2}(\vartheta) = -k_2 \tilde{\x}
\end{align*}
and when $k_2 > k_1$, we obtain the additional contributing points
\begin{align*}
\aalpha_{\mathrm{L}_1} &= \left(\sqrt{k_2^2 - k_1^2}, -k_1\right), \ \aalpha_{\mathrm{L}_2} = \left(-k_1, -\sqrt{k_2^2 - k_1^2} \right),
\end{align*}
as illustrated in Fig. \ref{fig:Ch5Singularities04}.
The bridge and arrow configuration at these contributing points is as displayed in Fig. \ref{fig:Ch5Singularities01}.
\begin{figure}[h!]
\centering
\includegraphics[width=\textwidth]{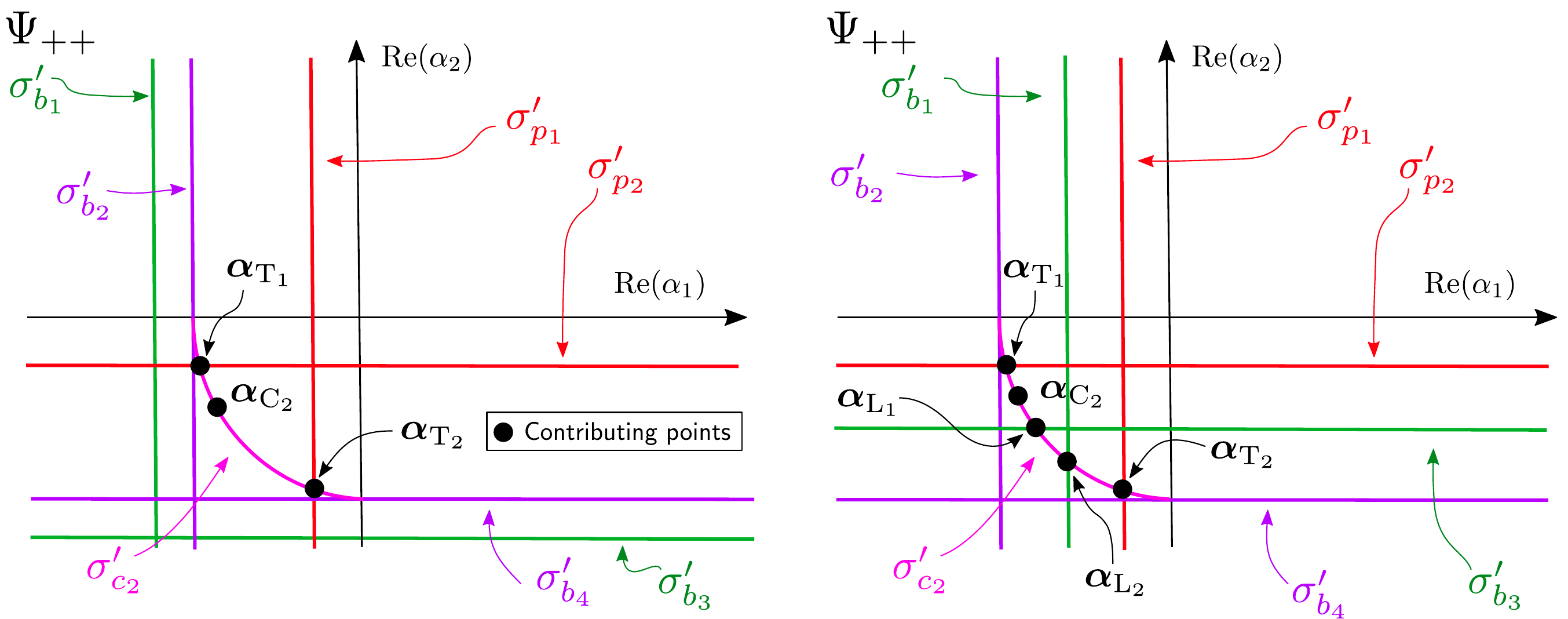}
\caption{Contributing points of $\Psi_{++}$ for $k_2 < k_1$ (left) and $k_2 > k_1$ (right).}
\label{fig:Ch5Singularities04}
\end{figure}

\subsubsection{The transmitted waves}\label{sec:FirstTrans}

Since the  wave components corresponding to the points $\aalpha_{\mathrm{T}_1}$ and $\aalpha_{\mathrm{T}_2}$ are computed similarly, we will only do it in detail for the wave corresponding to the point $\aalpha_{\mathrm{T}_1}$. We refer to \citep{KunzThesis} for a more detailed discussion. We henceforth assume that $k_2^2 - \a_2^2 \geq 0$, for otherwise $\aalpha_{\mathrm{T}_1} \notin \mathbb{R}^2$. In this case, we obtain no corresponding  wave component, which corresponds to total (internal) reflection of $\phi_{\iin}$ on the wedge's face $\{x_1=0, \ x_2 > 0\}$. Similarly, we assume $k_2^2 - \a_1^2 \geq 0$, for $k_2^2 - \a_1^2 < 0$ corresponds to total (internal) reflection of $\phi_{\iin}$ on the wedge's face $\{x_1>0, \ x_2 = 0\}$.\\

\noindent \textbf{The transmitted wave $\psi_{\mathrm{T}_1}$.}  Let us consider the contribution of the point $\aalpha^{\star} = \aalpha_{\mathrm{T}_1}$, which is a transverse crossing of $\sigma'_{p_2}$ and $\sigma'_{c_2}$. The corresponding bridge and arrow configuration is displayed in Fig. \ref{fig:Ch5SimpleTransmitted}. Let $\sigma_1 = \sigma_{c_2}$ and  $\sigma_{2} = \sigma_{p_2}$. 

\textbf{Step 1.} Choose $g_1(\aalpha) = k_2^2 - \alpha_1^2 - \alpha_2^2$ and $g_2(\aalpha) = \alpha_2 - \a_2$. Due to the bridge and arrow configuration at $\sigma'_1 \cap \sigma'_2$ and since $\nn_1 = - \boldsymbol{e}_r,$ and $\nn_2 = \boldsymbol{e}_{\Re(\alpha_2)}$, the sign factors are given by $s_1 = s_2 = +1$.

\textbf{Step 2.} For $\varkappa>0$, we have $\aalpha_{\mathrm{T}_1} \in \LHP \times \LHP$, and we use \eqref{eq.Ch5PsiCont2} for computing the asymptotic behaviour of $\Psi_{++}$ as $\aalpha \to \aalpha_{\mathrm{T}_1}$. Now, only the external additive term in equation \eqref{eq.Ch5PsiCont} contributes to the far-field since the integral term does not exhibit a crossing of singularities at $\aalpha_{\mathrm{T}_1}$. Therefore
\begin{align*}
\Psi_{++}(\aalpha)  \contr & - \frac{K_{- \circ}(\aalpha) (k_1^2 - \alpha_1^2 -\alpha_2^2)}{ K_{- \circ}(\a_1, \a_2) (\alpha_1 - \a_1)} \frac{1}{(\alpha_2 - \a_2)(k_2^2 - \alpha_1^2 - \alpha_2^2)} \\
\psim & \frac{4 \a_1 \sqrt{k_2^2 - \a_2^2}}{\sqrt{k_2^2 - \a_2} - \a_1} \frac{1}{(\alpha_2 - \a_2)(k_2^2 - \alpha_1^2 - \alpha_2^2)}, \numberthis
\end{align*}
which is exactly of the form 
\begin{align*}
\Psi_{++}(\aalpha) \sim A \times g_1(\aalpha)^{-m_1} \times g_2(\aalpha)^{-m_2}
\end{align*}
for 
\begin{align*}
A = \frac{4 \a_1 \sqrt{k_2^2 - \a_2^2}}{\sqrt{k_2^2 - \a_2^2} -\a_1}, \quad \text{ and } \quad m_1 = m_2 = 1.
\end{align*}

\textbf{Step 3.} We can now use equation \eqref{eq.Ch5EstimateCrossing} to obtain the  wave component $\psi_{\mathrm{T}_1}$ corresponding to the crossing $\aalpha_{\mathrm{T}_1}$, and find
\begin{align*}
\psi_{\mathrm{T}_1}(\x) 
&= \frac{2 \a_1}{\a_1 -\sqrt{k^2_2 - \a_2^2}}
e^{- i (-x_1  \sqrt{k_2^2 - \a_2^2} + x_2 \a_2)} \mathcal{H}(x_1) \mathcal{H}\left(x_1 \a_2 + x_2  \sqrt{k_2^2 - \a_2^2} \right). \numberthis \label{eq.Ch5psiT1}
\end{align*}

\noindent \textbf{The transmitted wave $\psi_{\mathrm{T}_2}$.} Now consider the point $\aalpha^{\star} = \aalpha_{\mathrm{T}_2}$. The corresponding  wave component is computed similarly: it is again sufficient to only analyse \eqref{eq.Ch5PsiCont}'s external additive term and, again, we can use \eqref{eq.Ch5EstimateCrossing} to find the  wave component  $\psi_{\mathrm{T}_2}$  corresponding to the crossing $\aalpha_{\mathrm{T}_2}$. This yields
\begin{align*}
\psi_{\mathrm{T}_2}(\x) 
& = \frac{ 2 \a_2}{\a_2 - \sqrt{k_2^2 - \a_1^2}} e^{- i (x_1 \a_1 - x_2 \sqrt{k_2^2 - \a_1^2})} \mathcal{H}\left( x_1  \sqrt{k_2^2 - \a_1^2} + x_2 \a_1\right) \mathcal{H}(x_2). \label{eq.transmissionCoeff} \numberthis 
\end{align*}
The fields $\psi_{\mathrm{T}_1}$ and $\psi_{\mathrm{T}_2}$ are in perfect agreement with the corresponding wave-fields that are predicted by GO, and which are illustrated in Fig. \ref{fig:Wedge04}.

\subsubsection{The cylindrical diffracted wave}

Next, we investigate the contributing behaviour of $\Psi_{++}$ near  $\aalpha^{\star}= \aalpha_{\mathrm{C}_2}(\vartheta) = -k_2\tilde{\x}$, which is an isolated SOS on $\sigma'_{c_2}$ with respect to $\tilde{\x}$. The bridge and arrow configuration corresponding to this SOS is displayed in Fig. \ref{fig:PsiCW}.

\textbf{Step 1.} Choose $g(\aalpha) = k_2^2 - \alpha_1^2 - \alpha_2^2$ as the defining function for $\sigma_{c_2}$. By \citep{AssierShaninKorolkov2022} Appendix A, we then find 
$
\xi = 1/(4 k_2^2),
$
where $\xi$ is the constant introduced in Section \ref{sec:Ch5IntegralEstimation}. Since $\tilde{\x} = \nn$, and since the bridge and arrow configuration is as displayed in Fig. \ref{fig:PsiCW}, the sign factor is given by $s = +1$. 

\textbf{Step 2.} To compute the asymptotic behaviour of $\Psi_{++}$ as $\aalpha \to \aalpha_{\mathrm{C}_2}(\vartheta)$, we use the fact that $\Psi_{++} = \Phi/K $, where $\Phi  = -\Phi_{3/4} - P_{++}$, and that $\Phi$ is regular at $\aalpha_{\mathrm{C}_2}(\vartheta)$, as can be seen from Fig. \ref{fig:Ch5Singularities01}. Thus, as $\aalpha \to \aalpha_{\mathrm{C}_2}(\vartheta)$, we have
\begin{align}
\Psi_{++}(\aalpha)
& \sim \Phi(\aalpha_{\mathrm{C}_2}(\vartheta)) \frac{k_1^2 - k_2^2}{k_2^2 - \alpha_1^2 - \alpha_2^2}
\end{align}
which is exactly of the form 
\begin{align*}
\Psi_{++}(\aalpha) \sim  A  \times g(\aalpha)^{-m_{}}
\end{align*}
for 	
\begin{align*}
A = \Phi(\aalpha_{\mathrm{C}_2}(\vartheta)) \left(k_1^2 - k_2^2\right), \quad \text{ and } \quad m_{} = 1.
\end{align*}

\textbf{Step 3.} We can now use equation \eqref{eq.Ch5EstimateSOS} and find that the  wave component $\psi_{\mathrm{C}}$ corresponding to $\aalpha_{\mathrm{C}_2}(\vartheta)$ is given by
\begin{align}
\psi_{\mathrm{C}}(r, \vartheta) =  \frac{\Phi(k_2 \cos(\vartheta), k_2 \sin(\vartheta))}{8 \sqrt{2 \pi}} \left(\frac{k_1^2}{k_2^2} -1\right) e^{- i 3 \pi/4 } \times \frac{e^{-i k_2 r}}{\sqrt{k_2 r}}, \label{eq.PsiDiffracted}
\end{align}
which is supported only for $\vartheta \in (0, \pi/2)$. The quantity
\begin{align}
D_{\psi_{\mathrm{C}}}(\vartheta, \vartheta_{0}) = \frac{\Phi(k_2 \cos(\vartheta), k_2 \sin(\vartheta))}{8 \sqrt{2 \pi}} \left(\frac{k_1^2}{k_2^2} -1\right) e^{- i 3 \pi/4 }
\end{align}
corresponds to the (cylindrical) \emph{diffraction coefficient} of $\psi$.

\subsubsection{The lateral diffracted waves}\label{sec:FirstLateral}

When $k_2 >k_1$, the waves $\psi_{\mathrm{T}_1}, \ \psi_{\mathrm{T}_2},$ and $\psi_{\mathrm{C}_2}$ are the only  wave components of $\psi$. However, when $k_1 < k_2$, we obtain two additional waves corresponding to the contributing points $\aalpha_{\mathrm{L}_1}$ and $\aalpha_{\mathrm{L}_2}$, as shown in Fig. \ref{fig:Ch5Singularities04}, right. Thus, let $k_1 < k_2$. Moreover, let us assume that $k_1 \neq \sqrt{2} k_2$, for if $k_1 = \sqrt{2}k_2$, the points $\aalpha_{\mathrm{L}_2}$ and $\aalpha_{\mathrm{L}_1}$ coalesce, thereby leading to a triple crossing of singularities. We expect that the formulae derived in this Section remain valid in this case, but the detailed study of this will be the basis for future work.\\

\noindent \textbf{The lateral wave $\psi_{\mathrm{L}_2}$.}  We begin by studying the contributing behaviour of $\Psi_{++}$ near the point $\aalpha^{\star} = \aalpha_{\mathrm{L}_2}$, which is a transverse crossing of $\sigma'_{b_1}$ and $\sigma'_{c_2}$, with bridge and arrow configuration as displayed in Fig. \ref{fig:Ch5LateralTransmitted}. Let $\sigma_1 = \sigma_{b_1}$ and $\sigma_{2} = \sigma_{c_2}$. 

\textbf{Step 1.} Choose $g_1(\aalpha) = \alpha_1$ and $g_2(\aalpha) = k_2^2 - \alpha_1^2 - \alpha_2^2$. As before, the sign factors are given by $s_1 = s_2 = +1$.

\textbf{Step 2.} For $\varkappa >0$, we have $\aalpha_{\mathrm{L}_2} \in \LHP \times \LHP$. However, since the integral terms in formulae \eqref{eq.Ch5PsiCont} and \eqref{eq.Ch5PsiCont2} are singular at $\aalpha_{\mathrm{L}_2}$, we cannot simply study the contributing behaviour of these formulae's additive terms, but instead the integral terms need to be accounted for as well. This turns out to be easiest when using formula \eqref{eq.Ch5PsiCont2}, wherein the integral term exhibits a polar singularity as $\alpha_1 \to -k_1$. Therefore, we need to modify formula \eqref{eq.Ch5PsiCont2} to study this limit. To this end let us assume that $\aalpha$ is close to $\aalpha_{\mathrm{L}_2}$, and change the contour $P$ in \eqref{eq.Ch5PsiCont2} to some contour $P'$ as illustrated in Fig. \ref{fig:Ch5LateralContour031}. 
\begin{figure}[h!]
\centering
\includegraphics[width=\textwidth]{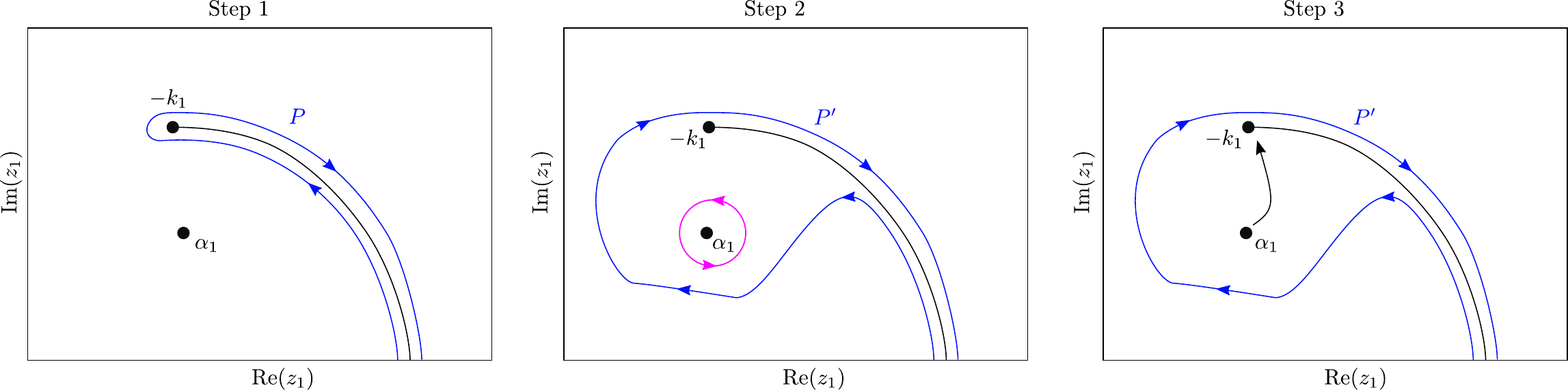}
\caption{Contour change that modifies equation \eqref{eq.Ch5PsiCont2} to allow for taking the limit $\alpha_1 \to -k_1$.}
\label{fig:Ch5LateralContour031}
\end{figure}
This picks up a residue of the integrand in \eqref{eq.Ch5PsiCont2} relative to counter-clockwise orientation.
Overall, we then obtain 
\begin{align*}
\Psi_{++}(\aalpha) =&  \frac{-i}{4 \pi K_{+ \circ}(\aalpha)} \ \int_{P'}\mathcal{I}_1(\aalpha; z_1)  dz_1  + \frac{ 1 }{2 K(\aalpha)} 
\left(
\frac{\left(k^2_2 - k_1^2 \right) \Psi_{++}\left(\alpha_1, \mmysqrt{k^2_1 -\alpha^2_1}\right)}{\left(\mmysqrt{k^2_1 -\alpha^2_1} - \alpha_2\right)\mmysqrt{k^2_1 -\alpha^2_1}}
\right) \\
& -  	\frac{K_{\circ-}(\mathfrak{a}_1,\alpha_2)P_{++}(\aalpha)}{K_{+\circ}(\aalpha) K_{-\circ}(\mathfrak{a}_1,\alpha_2)K_{\circ-}(\mathfrak{a}_1,\mathfrak{a}_2)}. \numberthis \label{eq.Lateral1.01}
\end{align*}
Now, the integral term as well as the second additive term in \eqref{eq.Lateral1.01} does not exhibit a crossing of singularities at $\aalpha = \aalpha_{\mathrm{L}_2}$ and therefore, as $\aalpha \to \aalpha_{\mathrm{L}_2}$, we obtain
\begin{align}
\Psi_{++}(\aalpha) \contr  \frac{ 1 }{2 K(\aalpha)} 
\left(
\frac{\left(k^2_2 - k_1^2 \right) \Psi_{++}\left(\alpha_1, \mmysqrt{k^2_1 -\alpha^2_1}\right)}{\left(\mmysqrt{k^2_1 -\alpha^2_1} - \alpha_2\right)\mmysqrt{k^2_1 -\alpha^2_1}}
\right)\! \cdot \label{eq.Lateral1.02}
\end{align}
Let us now analyse the behaviour of $\Psi_{++}(\alpha_1, \mmysqrt{k^2_1 -\alpha^2_1})$ as $\alpha_1 \to -k_1$. Since for $\varkappa >0$ we have $(\alpha_1, \mmysqrt{k^2_1 -\alpha^2_1}) \in \LHP \times \UHP,$ we need to use \eqref{eq.Ch5PsiCont} to analyse this limit, so
\begin{multline}
K_{\circ +}\left(\alpha_1, \mmysqrt{k_1^2-\alpha_1^2}\right) \Psi_{++}\left(\alpha_1,\mmysqrt{k_1^2-\alpha_1^2}\right)  
=
\frac{-i}{4 \pi } \ \int_{P} \mathcal{I}_2\left(\alpha_1, \mmysqrt{k_1^2 - \alpha_1^2}; z_2\right) dz_2\\
-  
\frac{K_{-\circ}(\alpha_1,\mathfrak{a}_2)P_{++}\left(\alpha_1, \mmysqrt{k_1^2-\alpha_1^2}\right)}{ K_{\circ-}(\alpha_1,\mathfrak{a}_2)K_{-\circ}(\mathfrak{a}_1,\mathfrak{a}_2)}\cdot \label{eq.Lateral1.03}
\end{multline}
Now, using \eqref{eq.I2Def}, one can show that the term on the right hand side of \eqref{eq.Lateral1.03} is regular as $\alpha_1 \to -k_1$. We can hence define the quantity $\mathcal{L}_{\psi_{\mathrm{L}_2}}(\vartheta_{0})$ by
\begin{align}
\mathcal{L}_{\psi_{\mathrm{L}_2}}(\vartheta_{0}) = \lim_{\alpha_1 \to -k_1}\left(	K_{\circ +}\left(\alpha_1, \mmysqrt{k_1^2 - \alpha_1^2}\right) \Psi_{++}\left(\alpha_1, \mmysqrt{k_1^2 - \alpha_1^2}\right) \right).
\end{align}
Note that, $\mathcal{L}_{\psi_{\mathrm{L}_2}}(\vartheta_{0})$ depends on the incident angle $\vartheta_{0}$, since the spectral function $\Psi_{++}$ exhibits this dependence (although it has been suppressed for brevity).

We therefore
obtain, as $\alpha_1 \to -k_1$, 	
\begin{align}
\Psi_{++}\left(\alpha_1,\mmysqrt{k_1^2-\alpha_1^2}\right) 
& \sim \frac{\mathcal{L}_{\psi_{\mathrm{L}_2}}(\vartheta_{0})}{K_{\circ +}\left(\alpha_1, \mmysqrt{k_1^2-\alpha_1^2}\right)},
\end{align}
and thus, as $\aalpha \to \aalpha_{\mathrm{L}_2}$, we find from \eqref{eq.Lateral1.02} that
\begin{align*}
\Psi_{++}(\aalpha) \contr  \frac{ 1 }{2 K(\aalpha)} 
&	\left(
\frac{\left(k^2_2 - k_1^2 \right)}{\left(\mmysqrt{k^2_1 -\alpha^2_1} - \alpha_2\right)\mmysqrt{k^2_1 -\alpha^2_1}}
\frac{\mathcal{L}_{\psi_{\mathrm{L}_2}}(\vartheta_{0})}{K_{\circ +}\left(\alpha_1, \mmysqrt{k_1^2-\alpha_1^2}\right)} 
\right)\! \cdot \numberthis
\end{align*}
Now, a lengthy but straightforward computation shows that 
\begin{align}
\frac{\left(k^2_2 - k_1^2 \right)}{\left(\mmysqrt{k^2_1 -\alpha^2_1} - \alpha_2\right)\mmysqrt{k^2_1 -\alpha^2_1}}
\frac{1}{K_{\circ +}\left(\alpha_1, \mmysqrt{k_1^2-\alpha_1^2}\right)}  
\contr \frac{- 4 \mmysqrt{k_2^2 - k_1^2}\sqrt{2k_1} \mmysqrt{k_1+ \alpha_1}}{k_2^2-k_1^2} 
\end{align}	
and therefore, we find 
\begin{align*}
\Psi_{++}(\aalpha) \contr & 
2\sqrt{k_2^2 - k_1^2}\sqrt{2k_1} \mathcal{L}_{\psi_{\mathrm{L}_2}} (\vartheta_{0}) \times \frac{g_1(\aalpha)^{1/2}}{g_2(\aalpha)}\cdot \numberthis \label{eq.Lateral1.04}
\end{align*}
This is exactly of the form
\begin{align*}
\Psi_{++}(\aalpha) \contr A \times g_1(\aalpha)^{-m_1} \times g_2(\aalpha)^{-m_2}
\end{align*}
with $m_1 = -1/2, \ m_2 = 1$, and 
\begin{align*}
A  = 2\sqrt{k_2^2 - k_1^2}\sqrt{2k_1} \mathcal{L}_{\psi_{\mathrm{L}_2}}(\vartheta_{0}).
\end{align*}

\textbf{Step 3.} Let us introduce the \emph{lateral diffraction coefficient} 
\begin{align*}
D_{\psi_{\mathrm{L}_2}}(\vartheta_{0}) &=  \frac{2(k_2^2 -k_1^2)^{3/4} \sqrt{k_1} e^{ i 3\pi/4} \mathcal{L}_{\psi_{\mathrm{L}_2}}(\vartheta_{0})}{ \sqrt{\pi}}\cdot
\end{align*}
Then, using \eqref{eq.Ch5EstimateCrossing} to obtain the  wave component $\psi_{\mathrm{L}_2}$ corresponding to the crossing at $\aalpha_{\mathrm{L}_2}$, we find that
\begin{align*}
\psi_{\mathrm{L}_2}(\x)  = & D_{\psi_{\mathrm{L}_2}}(\vartheta_{0})\frac{e^{i ( k_1 x_1  + \sqrt{k_2^2 - k_1^2} x_2)}}{|x_1 \sqrt{k^2_2 -k_1^2} - x_2 k_1|^{3/2}} \mathcal{H}\left(\sqrt{k_2^2 - k_1^2}x_1 - k_1 x_2\right)\mathcal{H}(x_2).\numberthis \label{eq.Lateral1.Wave} 
\end{align*}

\noindent \textbf{The lateral wave $\psi_{\mathrm{L}_1}$.} We now consider the point $\aalpha^{\star} = \aalpha_{\mathrm{L}_1}$. Using formula \eqref{eq.Ch5PsiCont} to obtain the contributing asymptotic behaviour of $\Psi_{++}$ as $\aalpha \to \aalpha_{\mathrm{L}_1}$, all corresponding computations are analogous to those carried out previously for the point $\aalpha_{\mathrm{L}_2}$, so we only give the result. Again, we refer to \citep{KunzThesis} for a more detailed discussion. Particularly, we find that the function $
K_{+ \circ}(\mmysqrt{k_1^2 - \alpha_2^2}, \alpha_2) \Psi_{++}(\mmysqrt{k_1^2 - \alpha_2^2}, \alpha_2)
$
is regular as $\alpha_2 \to -k_1$, and we can hence define the quantity $\mathcal{L}_{\psi_{\mathrm{L}_1}}(\vartheta_{0})$ via 
\begin{align}
\mathcal{L}_{\psi_{\mathrm{L}_1}}(\vartheta_{0}) = \lim_{\alpha_2 \to -k_1} \left(	K_{+ \circ}\left(\mmysqrt{k_1^2 - \alpha_2^2}, \alpha_2\right) \Psi_{++}\left(\mmysqrt{k_1^2 - \alpha_2^2}, \alpha_2\right)\right).
\end{align}
Then, upon introducing the lateral diffraction coefficient
\begin{align*}
D_{\psi_{\mathrm{L}_1}}(\vartheta_{0}) = \frac{ 2(k_2^2 -k_1^2)^{3/4} \sqrt{k_1} e^{ i 3\pi/4} \mathcal{L}_{\psi_{\mathrm{L}_1}}(\vartheta_{0})}{ \sqrt{\pi}} 
\end{align*}
we find that
\begin{align}
\psi_{L 1}(\x) = D_{\psi_{\mathrm{L}_1}}(\vartheta_0)
\frac{e^{i ( \sqrt{k_2^2 - k_1^2} x_1  + k_1 x_2)}}{|x_2 \sqrt{k^2_2 -k_1^2} -x_1 k_1  |^{3/2}} \mathcal{H}\left(k_1 x_1 - \sqrt{k_2^2 -k_1^2} x_2 \right) \mathcal{H}(x_1).
\label{eq.Lateral2.Wave}
\end{align}

Formulae \eqref{eq.Lateral1.Wave} and \eqref{eq.Lateral2.Wave} are in perfect agreement with the form of lateral waves we expect from the studies outlined in Section \ref{sec:LateralInformal}. When $k_2 < k_1$, we set $\psi_{\mathrm{L}_1} = \psi_{\mathrm{L}_2} \equiv 0$.

\subsection{Wave components of $\phi$}\label{sec:FarFieldPhi}	

We now proceed to study the remaining transversal crossings as well as the points which are isolated SOS of $\Phi_{3/4}$.
When $k_1 < k_2$, these (potentially) contributing points are given by 
\begin{align*}
\aalpha_{\mathrm{R}_1} = \left(-\sqrt{k_1^2 - \a_2} , \a_2\right), \ 	\aalpha_{\mathrm{R}_2} = \left(\a_1,  \sqrt{k_1^2 - \a_1^2}\right), \ \aalpha_{\mathrm{C}_1}(\tilde{\x}) = - k_1 \tilde{\x},
\end{align*}
and, when $k_1 > k_2$, we obtain the additional contributing points
\begin{align*}
&\aalpha_{\mathrm{L}_1} = \left(-k_2, \sqrt{k_1^2 -k_2^2}\right), \ \aalpha_{\mathrm{L}_2} = \left(\sqrt{k_1^2 -k_2^2}, -k_2\right).
\end{align*}
These points are shown in Fig. \ref{fig:Ch5Singularities03}, and the corresponding bridge and arrow configurations are as displayed in Fig. \ref{fig:Ch5Singularities01}. Here, we have re-used the notation $\aalpha_{\mathrm{L}_1}$ and $\aalpha_{\mathrm{L}_2}$, which we already used for crossings of singularities of $\Psi_{++}$. The rationale for this choice will be explained in Section \ref{sec:Ch5LateralPhi}.

\begin{figure}[h!]
\centering
\includegraphics[width=\textwidth]{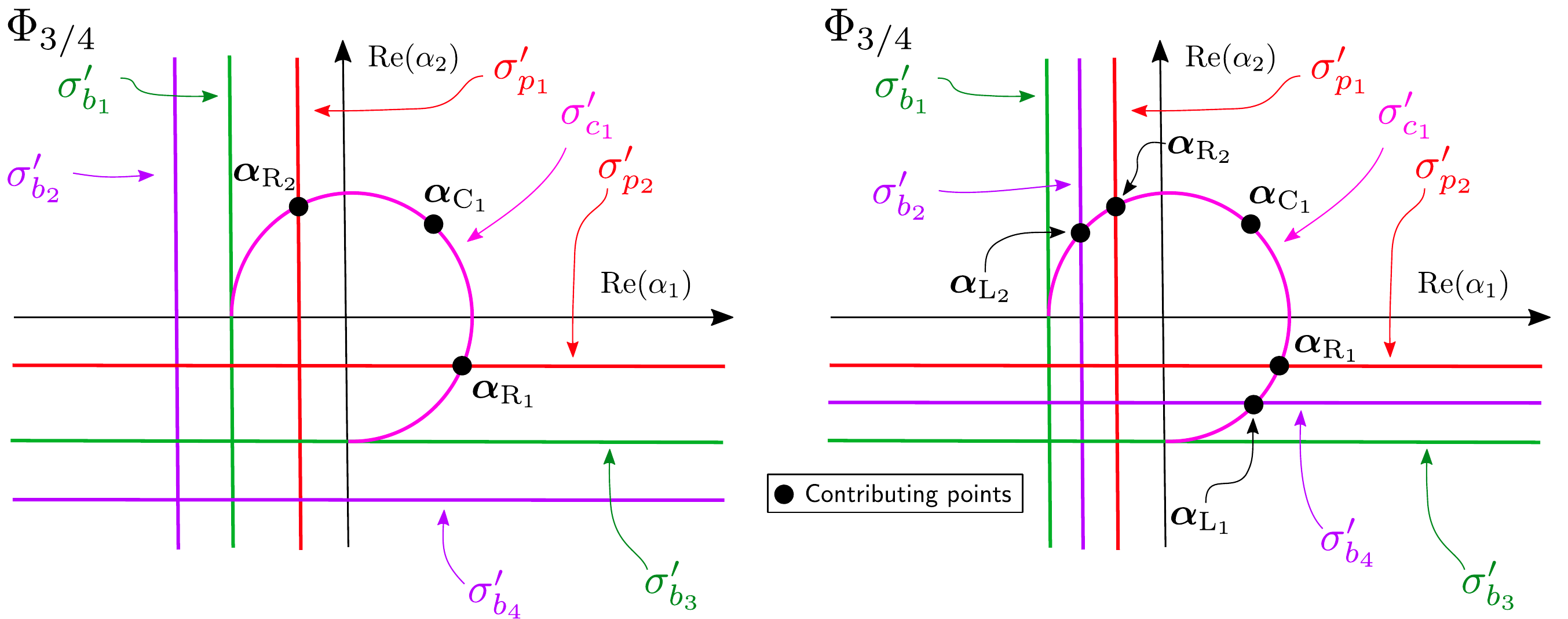}
\caption{Contributing points of $\Phi_{3/4}$ for $k_1 < k_2$ (left) and $k_1 > k_2$ (right).}
\label{fig:Ch5Singularities03}
\end{figure}

We often omit the detailed calculations in the following analysis, as they are similar to calculations shown previously. The details can be found in \citep{KunzThesis}, Chapter 5.

\subsubsection{The reflected waves}\label{sec:FirstReflectphi}
The reflected waves $\phi_{\mathrm{R}_1}$ and $\phi_{\mathrm{R}_2}$ corresponding to the points $\aalpha_{\mathrm{R}_1}$ and $\aalpha_{\mathrm{R}_2}$ are obtained similarly to how the transmitted wave $\psi_{\mathrm{T}_1}$ was obtained in Section \ref{sec:FirstTrans}, so we omit the details (only the additive external term in \eqref{eq.Ch5PsiCont2} needs to be considered). 

One then finds 
\begin{align}
\phi_{\mathrm{R}_1}(\x) = 	\frac{\a_1 + \sqrt{k^2_2 - \a^2_2} }{\a_1 - \sqrt{k_2^2-\a^2_2}} e^{-i(-\mathfrak{a}_1x_1 + \mathfrak{a}_2 x_2)}\mathcal{H}(-x_1 \a_2  - x_2 \a_1) \mathcal{H}(-x_1). \label{eq.SimpleRefl.02}
\end{align}
and 
\begin{align*}
\phi_{\mathrm{R}_2}(\x) 
& =  \frac{\mathfrak{a}_2 + \sqrt{k^2_2 - \a^2_1}}{\mathfrak{a}_2 - \sqrt{k^2_2 - \a^2_1}} e^{-i (x_1 \a_1 - x_2 \a_2)} \mathcal{H}\left(-x_1 \a_2 -  x_2 \a_1 \right) \mathcal{H}(-x_2). \numberthis \label{eq.ReflectionCoeff}
\end{align*}
These formulae agree perfectly with the corresponding wave-fields predicted by GO (which are displayed in Fig. \ref{fig:Wedge04}).

\subsubsection{The cylindrical diffracted wave}\label{sec:PhiDiffr}	

Let us now study the contributing behaviour of $\Phi$ near $\aalpha^{\star} = \aalpha_{\mathrm{C}_1}(\vartheta) = -k_1 \tilde{\x}$. This is an isolated SOS on $\sigma'_{c_1}$ with respect to $\tilde{\x}$. The bridge and arrow configuration is displayed in Fig. \ref{fig:SaddleSC}. Observing that $\Psi_{++}$ is regular at $\aalpha_{\mathrm{C}_1}(\vartheta)$, the  wave component $\phi_{\mathrm{C}}$ corresponding to the SOS $\aalpha_{\mathrm{C}_1}(\vartheta)$ can be computed similarly to how the wave $\psi_{\mathrm{C}}$ was computed, so we omit the details. We find that 
\begin{align}
\phi_{\mathrm{C}}(\x) & =  \frac{\Psi_{++}(k_1 \cos(\vartheta), k_1 \sin(\vartheta))}{8 \sqrt{2 \pi}} \left(1 - \frac{k_2^2}{k_1^2} \right) e^{- i 3 \pi /4} \times \frac{e^{-i k_1 r}}{\sqrt{k_1 r}}. \label{eq.PhiDiffracted} 
\end{align}
The quantity
\begin{align*}
D_{\phi_{\mathrm{C}}}(\vartheta, \vartheta_{0}) = \frac{\Psi_{++}(k_1 \cos(\vartheta), k_1 \sin(\vartheta))}{8 \sqrt{2 \pi}} \left(1 - \frac{k_2^2}{k_1^2} \right) e^{- i 3 \pi /4}
\end{align*}
corresponds to the (cylindrical) diffraction coefficient of $\phi_{\ssc}$.

\subsubsection{The lateral diffracted waves}\label{sec:Ch5LateralPhi}

If $k_1 < k_2$, the  wave components of $\phi_{\ssc}$ are given by $\phi_{\mathrm{R}_1}, \ \phi_{\mathrm{R}_2},$ and $\phi_{\mathrm{C}_1}$. In this case, the  wave components of $\psi$ include the lateral waves $\psi_{\mathrm{L}_1}$ and $\psi_{\mathrm{L}_2}$. However, when $k_1 > k_2$, these lateral waves are present on the wedges exterior and correspond to the crossings $\sigma'_{b_2} \cap \sigma'_{c_1}$ and $\sigma'_{b_4} \cap \sigma'_{c_1}$, respectively, as we will prove in this section. We label these crossings as $\aalpha_{\mathrm{L}_2}$ and $\aalpha_{\mathrm{L}_1}$, respectively. This agrees with how we labelled the crossings of $\sigma'_{b_3}$ and $\sigma'_{b_1}$ with $\sigma'_{c_2}$, which led to the lateral waves on the wedge's interior. However, there is no ambiguity whether these points correspond to crossings of singularities of $\Phi_{3/4}$, or $\Psi_{++}$: whenever $k_1 < k_2$, these points correspond to crossings of singularities of $\Psi_{++}$ and whenever $k_2 < k_1$, they correspond to crossings of singularities of $\Phi_{3/4}$. Thus, in either case, the points $\aalpha_{\mathrm{L}_1}$ and $\aalpha_{\mathrm{L}_2}$ correspond to the total wave-field's (only) two lateral waves. Let $k_1 > k_2$. \\

\noindent \textbf{The lateral wave $\phi_{\mathrm{L}_2}$.} Let us begin by studying the contributing behaviour of $\Phi_{3/4}$ near the point $\aalpha^{\star}= \aalpha_{\mathrm{L}_2}$. To this end, we study the contributing behaviour of $-K \Psi_{++} = - \Phi$. Since $\Phi_{3/4} = -\Phi - P_{++}$, and since $P_{++}$ does not exhibit a crossing of singularities at $\aalpha_{\mathrm{L}_2}$, the contributing behaviour of $-\Phi$ completely determines the contributing behaviour of $\Phi_{3/4}$. Let $\sigma_{1} = \sigma_{c_1}$ and $\sigma_2 = \sigma_{b_2}$. Similarly to Section \ref{sec:FirstLateral}, we assume that $k_2 \neq \sqrt{2}k_1$.

\textbf{Step 1.} Choose $g_1(\aalpha) = k_1^2 - \alpha_1^2 - \alpha_2^2$ and  $g_2(\aalpha) = \alpha_2 + k_2$. Since $\boldsymbol{e}_r = - \nn_1$ and $\boldsymbol{e}_{\Re(\alpha_1)} = \nn_2$, and since  the bridge and arrow configuration at $\aalpha_{\mathrm{L}_2}$ is as displayed in Fig. \ref{fig:Ch5LateralReflected}, we find that the sign factors are given by $s_1=s_2=+1$.

\textbf{Step 2.} For $\varkappa >0$, we have $\aalpha_{\mathrm{L}_2} \in \LHP \times \UHP$. Let us use the formula \eqref{eq.Ch5PsiCont2} to study the contributing asymptotic behaviour of $\Phi$. Then, just as in Section \ref{sec:FirstLateral}, we find that the contributing asymptotic behaviour of $\Phi$ is given by 
\begin{align}
-\Phi(\aalpha) \contr  - \frac{1}{2} 
\left(
\frac{\left(k^2_2 - k_1^2 \right) \Psi_{++}\left(\alpha_1, \mmysqrt{k^2_1 -\alpha^2_1}\right)}{\left(\mmysqrt{k^2_1 -\alpha^2_1} - \alpha_2\right)\mmysqrt{k^2_1 -\alpha^2_1}}
\right), \text{ as } \aalpha \to \aalpha_{\mathrm{L}_2}. \label{eq.Ch5LateralPhi02}
\end{align}
Now, for $\alpha_1$ close to $-k_2$, we have	
\begin{multline}
K_{\circ +}\left(\alpha_1, \mmysqrt{k^2_1 -\alpha^2_1}\right)\Psi_{++}\left(\alpha_1, \mmysqrt{k^2_1 -
	\alpha^2_1}\right)   = 
\frac{-i}{4 \pi} \int_{P}\mathcal{I}_2\left(\alpha_1, \mmysqrt{k^2_1 -\alpha^2_1}; z_2\right) dz_2 \\ - \frac{K_{-\circ}(\alpha_1,\mathfrak{a}_2)P_{++}\left(\alpha_1, \mmysqrt{k^2_1 -\alpha^2_1}\right)}{K_{\circ+}\left(\alpha_1, \mmysqrt{k^2_1 -\alpha^2_1}\right) K_{\circ-}(\alpha_1,\mathfrak{a}_2)K_{-\circ}(\mathfrak{a}_1,\mathfrak{a}_2)}\cdot \label{eq.Ch5LateralPhi03} \numberthis
\end{multline}	

Contrary to how the lateral waves on the wedges interior were analysed, the integral term in \eqref{eq.Ch5LateralPhi03} is non-regular at $\alpha_1 = -k_2$ (and so is the external additive term), so we cannot directly proceed as in Section \ref{sec:FirstLateral}. 

Now, a lengthy but straightforward calculation shows that
\begin{multline}
K_{\circ +}\left(\alpha_1, \mmysqrt{k_1^2 - \alpha_1^2}\right)\Psi_{++}\left(\alpha_1, \mmysqrt{k_1^2 - \alpha_1^2}\right) = \\ \frac{-i}{4 \pi}\int_P\left(\mathcal{I}_2\left(\alpha_1, \mmysqrt{k_1^2 - \alpha_1^2}; z_2\right) - \mathcal{I}_2\left(\alpha_1, - \mmysqrt{k_2^2 - \alpha_1^2}; z_2\right)\right) dz_2 
+ \frac{\Phi\left(\alpha_1, -\mmysqrt{k_2^2 - \alpha_1^2}\right)}{K_{\circ -}\left(\alpha_1, - \mmysqrt{k_2^2 - \alpha_1^2}\right)} \\
+ \frac{P_{++}\left(\alpha_1, - \mmysqrt{k_2^2 - \alpha_1^2}\right) - P_{++}\left(\alpha_1, \mmysqrt{k_1^2 - \alpha_1^2}\right)}{K_{\circ -}(\alpha_1, \a_2)}\cdot \label{eq.Blablabla} \numberthis
\end{multline}	
The first and the third term on the right hand side of \eqref{eq.Blablabla} can be shown to be regular as $\alpha_1 \to -k_2$. Therefore, using \eqref{eq.Ch5LateralPhi02}, we conclude that
\begin{align*}
\Phi(\aalpha) \contr -  &\frac{1}{2} \frac{(k_2^2 -k_1^2)}{\left(\mmysqrt{k_1^2 - \alpha_1^2} - \alpha_2\right)\mmysqrt{k_1^2 - \alpha_1^2}} \frac{1}{K_{\circ +}\left(\alpha_1, \mmysqrt{k_1^2 - \alpha_1^2}\right)} \frac{\Phi\left(\alpha_1, -\mmysqrt{k_2^2 - \alpha_1^2}\right)}{K_{\circ -}\left(\alpha_1, - \mmysqrt{k_2^2 - \alpha_1^2}\right)},
\end{align*}
$\text{ as } \aalpha \to \aalpha_{\mathrm{L}_2},$ and it can be shown, just as in Section \ref{sec:FirstLateral}, that 
the function 
$
\Phi(\alpha_1, -\mmysqrt{k_2^2 - \alpha_1^2})/K_{\circ -}(\alpha_1, - \mmysqrt{k_2^2 - \alpha_1^2})
$
is regular as $\alpha_1 \to -k_2$. We can thus define the quantity
\begin{align}
\mathcal{L}_{\phi_{\mathrm{L}_2}}(\vartheta_{0}) = \lim_{\alpha_1 \to -k_2} 	\frac{\Phi\left(\alpha_1, -\mmysqrt{k_2^2 - \alpha_1^2}\right)}{K_{\circ -}\left(\alpha_1, - \mmysqrt{k_2^2 - \alpha_1^2}\right)}\cdot
\end{align}
Proceeding as in Section \ref{sec:FirstLateral}, we find that
\begin{align*}
\Phi(\aalpha) \contr A \times g_1(\aalpha)^{-m_1} \times g_2(\aalpha)^{-m_2}, \text{ as } \aalpha \to \aalpha_{\mathrm{L}_2}
\end{align*}
with $m_1 = 1$, $m_2 = -1/2$ and
\begin{align*}
A = - 2 \sqrt{k_1^2 -k_2^2}\sqrt{2 k_2} \mathcal{L}_{\phi_{\mathrm{L}_2}}(\vartheta_{0}).
\end{align*}		

\textbf{Step 3.} We use \eqref{eq.Ch5EstimateCrossing} to compute the corresponding  wave component $\phi_{\mathrm{L}_1}$. Introducing the lateral diffraction coefficient 
\begin{align*}
D_{\phi_{\mathrm{L}_2}}(\vartheta_{0}) = - \frac{2(k_1^2 -k_2^2)^{3/4}\sqrt{k_2}e^{i 3 \pi/4}  \mathcal{L}_{\phi_{\mathrm{L}_2}}(\vartheta_{0})}{\sqrt{\pi}}
\end{align*} 
we find that
\begin{align*}
\phi_{\mathrm{L}_2}(\x)  = D_{\phi_{\mathrm{L}_2}}(\vartheta_{0})	\frac{e^{i(k_2 x_1 - \sqrt{k_1^2 -k_2^2} x_2)}}{|x_1 \sqrt{k_1^2 -k_2^2} + x_2 k_2|^{3/2}} \mathcal{H}\left(\sqrt{k_1^2 - k_2^2} x_1 + k_2 x_2 \right) \mathcal{H}(-x_2).  \label{eq.Ch5phiL2}  \numberthis
\end{align*}

\noindent \textbf{The lateral wave $\phi_{\mathrm{L}_1}$.} It remains to study the contributing asymptotic behaviour of $\Phi$ near $\aalpha_{\mathrm{L}_1}$. This is done similarly to how the contributing asymptotics near $\aalpha_{\mathrm{L}_2}$ have been found, so we  omit the details.  Again, we can define the quantity 
\begin{align*}
\mathcal{L}_{\phi_{\mathrm{L}_1}}(\vartheta_{0}) = \lim_{\alpha_2 \to -k_2}	\frac{\Phi\left(-\mmysqrt{k_2^2 - \alpha_2}, \alpha_2\right)}{K_{-\circ}\left(-\mmysqrt{k_2^2 - \alpha_2}, \alpha_2\right)},
\end{align*}
and upon introducing the lateral diffraction coefficient
\begin{align}
D_{\phi_{\mathrm{L}_1}}(\vartheta_{0}) = - \frac{2(k_1^2 -k_2^2)^{3/4}\sqrt{k_2}e^{i 3 \pi/4} 	\mathcal{L}_{\phi_{\mathrm{L}_1}}(\vartheta_{0})}{\sqrt{\pi}},
\end{align}
we find that 
\begin{align}
\phi_{\mathrm{L}_1} = D_{\phi_{\mathrm{L}_1}}(\vartheta_0) \frac{e^{i(- \sqrt{k_1^2 -k_2^2} x_1 + k_2 x_2)}}{|x_1 k_2 + x_2 \sqrt{k_1^2 -k_2^2}|^{3/2}}\mathcal{H}\left(k_1 x_1 + \sqrt{k_1^2 -k_2^2} x_2\right)\mathcal{H}(x_2). \label{eq.Ch5lateralPhi2}
\end{align}

As in Section \ref{sec:FirstLateral}, formulae \eqref{eq.Ch5phiL2} and \eqref{eq.Ch5lateralPhi2} are in perfect agreement with the form of lateral waves we expect from the studies outlined in Section \ref{sec:LateralInformal}. When $k_1 < k_2$, we set $\phi_{\mathrm{L}_1} = \phi_{\mathrm{L}_2} \equiv 0$. 

We have thus proved the correctness of \eqref{eq.Ch1PWFarfield01} in the simple case, where the fields' GO components which are illustrated in Fig. \ref{fig:Wedge04} (left) are given by $\psi_{\GO} = \psi_{\mathrm{T}_1} + \psi_{\mathrm{T}_2}$ and $\phi_{\GO} = \phi_{\iin} + \phi_{\mathrm{R}_1}  + \phi_{\mathrm{R}_2}$. The diffracted far-field, consisting of cylindrical and lateral diffracted waves, is described by equations \eqref{eq.PsiDiffracted}, \eqref{eq.Lateral1.Wave}, \eqref{eq.Lateral2.Wave}, \eqref{eq.PhiDiffracted}, \eqref{eq.Ch5phiL2}, and \eqref{eq.Ch5lateralPhi2}, and is illustrated in Fig. \ref{fig:Ch5SimpleFarField}.

\begin{figure}[h!]
\centering
\includegraphics[width=\textwidth]{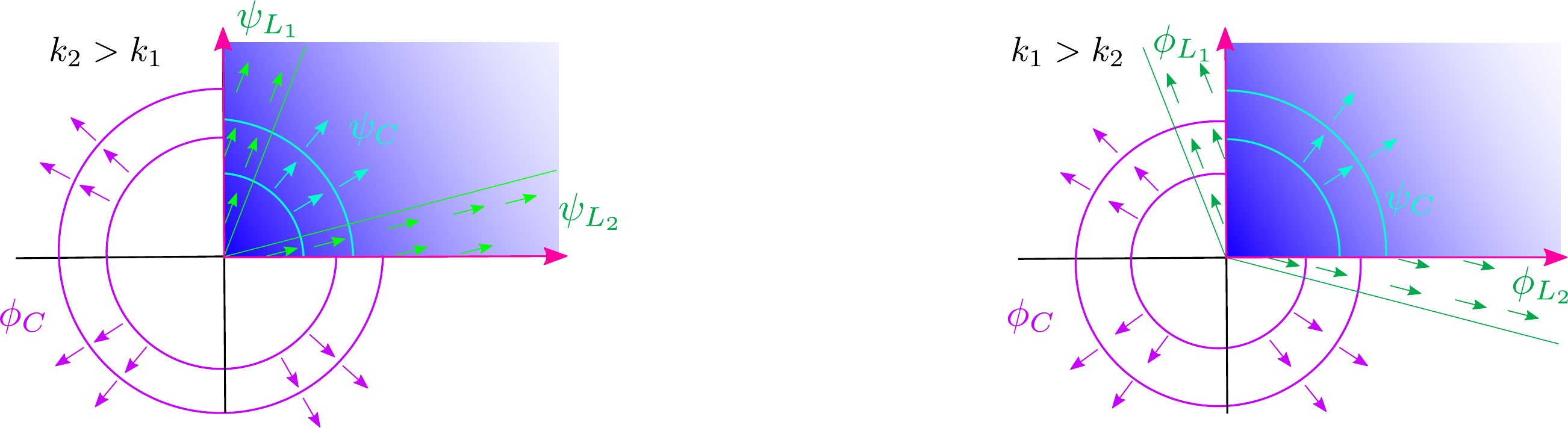}
\caption{Diffracted  wave components.}
\label{fig:Ch5SimpleFarField}
\end{figure}

\section{The complicated case}\label{sec:Complicated}		

Let us now consider the case $\vartheta_0 \in (\pi/2, \pi)$. Note that in this case, we cannot directly impose the Sommerfeld radiation condition via the limiting absorption principle, although, of course, a radiation condition needs to be imposed. The failure of defining this via the absorption principle is due to the fact that for positive imaginary part $\varkappa >0$ of $k_1$ and $k_2$, such incident angle changes the sign of $\a_2$: whereas in the simple case discussed in Section \ref{sec:Simple} we are guaranteed $\Im(\a_1), \Im(\a_2) < 0$ whenever $\varkappa >0$ we now have $\Im(\a_1) <0,$ and $\Im(\a_2) > 0$ whenever $\varkappa >0$.  We circumvent this issue by working directly in Fourier space. That is, we treat $\vartheta_0$ as an analytic parameter within the formulae for analytic continuation of $\Psi_{++}$ and thus obtain formulae for $\Psi_{++}$ when $\vartheta_0 \in (\pi/2, \pi)$. Here, by `analytic parameter' we mean that the spectral functions depend analytically on $\vartheta_{0}$. The radiation condition is then formulated by imposing the \emph{continuity of bypass}. If the arrow points to a given side of $\sigma_j'$ in the simple case, then it must point to this side of $\sigma_j'$ in the complicated  case. By uniqueness of analytic continuation and Stokes' theorem, this condition uniquely determines $\psi$ and $\phi_{\ssc}$ in the complicated case. Defining the physical fields in this way \emph{seems} to be a sensible way to impose the radiation condition for the following reasons. 
\begin{enumerate}[]
\item 1.  It leads to the GO components one would expect for the complicated case, as we will prove during this Section. \item 2. It prohibits the existence of any unphysical wave-fields.
\item 3. It implies that the physical fields depend continuously on the incident angle.
\end{enumerate}
However, note that the analytic dependence of $\Psi_{++}$ on $\vartheta_0$, of which the preceding points are a consequence, is, at this point, only an assumption. For a more detailed discussion involving the intricacies of formulating the radiation condition in the complicated case, we refer to \citep{AssierShaninKorolkov2023}, Section 2.

As mentioned in \citep{KunzAssierAC}, in the complicated case  we also obtain new singularities within the formulae for analytic continuation. Namely, the external additive term in \eqref{eq.Ch5PsiCont} becomes singular at $$\sigma_{sp} = \left\{\aalpha \in \mathbb{C}^2| \ \aalpha_1 \equiv  - \mmysqrt{k_2^2 - \a_2^2}\right\},$$ which, for $\vartheta_{0} \in (\pi/2,\pi)$, is a polar singularity. Here, the subscript $_{sp}$ stands for `secondary pole', since, as we shall see, it will lead to the existence of secondary reflected and transmitted waves. Note that for $\vartheta_0 = \pi$, we have $\sigma_{sp} = \sigma_{b_2}$. Therefore, imposing the aforementioned continuity of bypass, we find that the bridge and arrow configuration on $\sigma_{sp}$ is as displayed in Fig. \ref{fig:BridgeArrowCompl}. Moreover, the singularity $\sigma_{p_2}$ changes half-plane during such change of incident angle since the parameter $\a_2$ changes its sign (as is displayed in Fig. \ref{fig:BridgeArrowCompl}).  Note that formulae \eqref{eq.Ch5PsiCont} and \eqref{eq.Ch5PsiCont2} remain valid in the complicated case.	

The singularities for the complicated case and the corresponding bridge and arrow configurations are displayed in Fig. \ref{fig:BridgeArrowCompl}. We now discuss how this change of singularity structure affects the physical far-field. \\

\begin{figure}
\centering
\includegraphics[width=\textwidth]{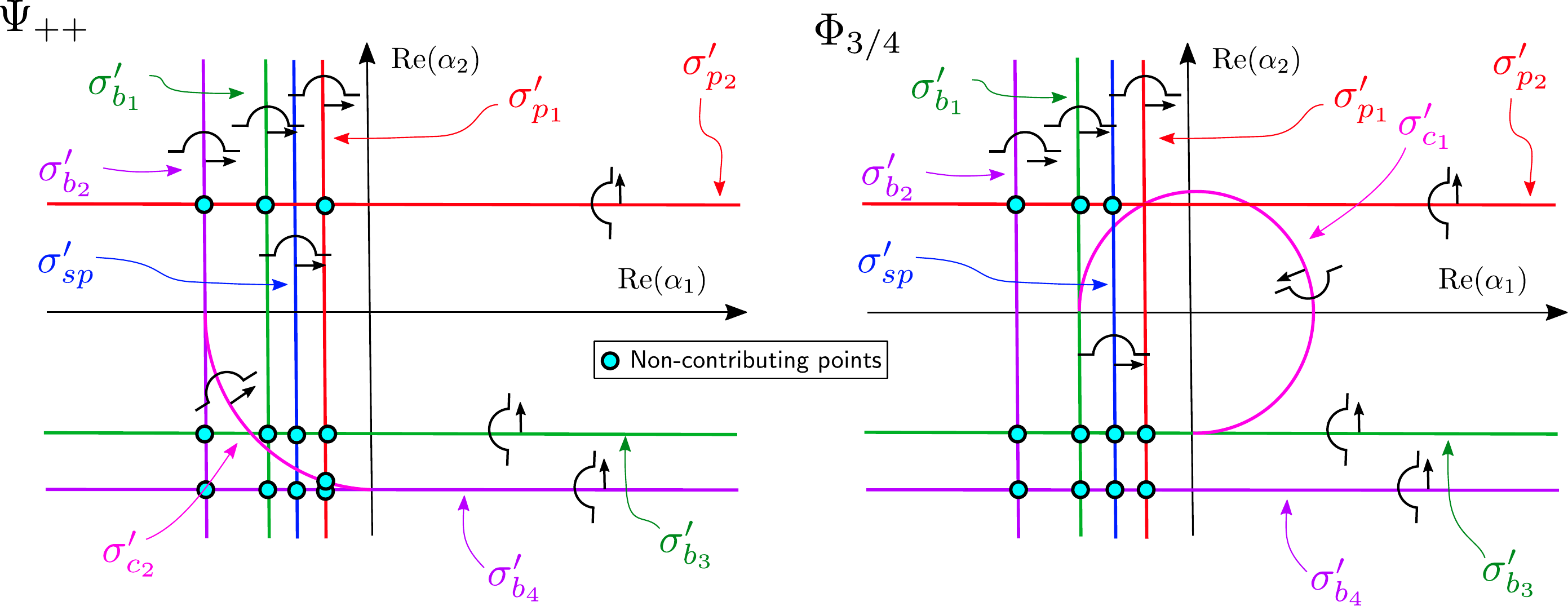}
\caption{Real traces of the spectral functions' irreducible singularities, corresponding bridge and arrow configuration, and non-contributing points.}
\label{fig:BridgeArrowCompl}	\label{fig:Ch5ComplNon-Contr}
\end{figure}

\noindent \textbf{The diffracted waves.}  The cylindrical and lateral diffracted waves are computed exactly as in the simple case, since the points $\aalpha_{\mathrm{C}_1}, \aalpha_{\mathrm{C}_2}$, $\aalpha_{\mathrm{L}_1}$, and $\aalpha_{\mathrm{L}_2}$ whose contributions yield these waves are unaffected by the change of incident angle. That is, formulae \eqref{eq.PsiDiffracted} (for $\psi_{\mathrm{C}}$), \eqref{eq.PhiDiffracted} (for $\phi_{\mathrm{C}}$), \eqref{eq.Lateral1.Wave} (for $\psi_{\mathrm{L}_2}$), \eqref{eq.Lateral2.Wave} (for $\psi_{\mathrm{L}_1}$), \eqref{eq.Ch5phiL2} (for $\phi_{\mathrm{L}_2}$), and \eqref{eq.Ch5lateralPhi2} (for $\phi_{\mathrm{L}_1}$) remain valid in the complicated case. \\

\noindent \textbf{Non-contributing points.} Similar to how the non-contributing points in the simple case were analysed in Section \ref{sec:Ch5Additive01}, it can be shown that the points highlighted in Fig. \ref{fig:Ch5ComplNon-Contr} are non-contributing.  \\

To obtain the  wave components of $\psi$ and $\phi$, it thus remains to consider the contributing behaviour of $\Psi_{++}$ at the points
\begin{align*}
\aalpha_{\mathrm{PT}} = \left(-\mmysqrt{k_2^2 - \a^2_2}, \a_2\right), \ \aalpha_{\mathrm{SR}} = \left(-\mmysqrt{k_2^2 - \a^2_2}, - \a_2 \right),
\end{align*}
and the contributing behaviour of $\Phi_{3/4}$ at the points
\begin{align*}
\aalpha_{\mathrm{PR}} = \left(-\a_1, \a_2\right), \ \aalpha_{\mathrm{SH}} =  \left(\a_1, \a_2\right), \ \aalpha_{\mathrm{ST}} = \left(-\mmysqrt{k_2^2 - \a_2^2}, \mmysqrt{k_1^2 - k_2^2 + \a_2^2} \right).
\end{align*}
Here, as in Section \ref{sec:InformalFarFieldWedge}, the subscripts $_{\mathrm{PT}}$, $_{\mathrm{SR}}$, $_{\mathrm{PR}}$, and $_{\mathrm{ST}}$ stand for `primary transmitted', `secondary reflected', `primary reflected', and `secondary transmitted', respectively. The subscript $_{\mathrm{SH}}$  corresponds to `shadow', as this crossing will yield the expected shadow region.

\begin{figure}
\centering
\includegraphics[width=\textwidth]{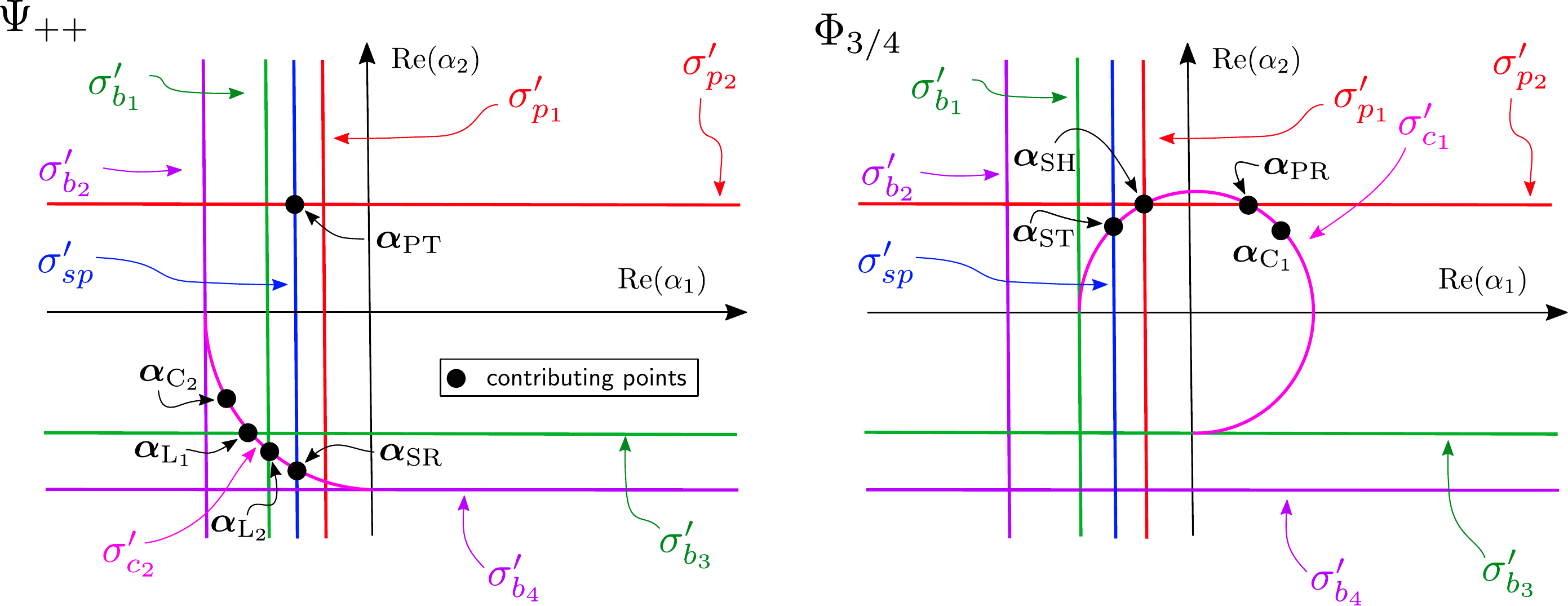}
\caption{Contributing points in the complicated case.}
\label{fig:Ch5ComplContr}
\end{figure}

The contributing points, including those which yield the diffracted waves, are, for $k_2 > k_1$, displayed in Fig. \ref{fig:Ch5ComplContr}.
Since all of the following computations are analogous to those done in the simple case, we will be brief.


\subsection{Geometrical optics components of $\psi$ and $\phi$}\label{sec:GOCompl} 

The computations yielding the wave-field's GO components in the complicated case are similar to those which yielded the transmitted wave $\psi_{\mathrm{T}_1}$: The corresponding  wave component can always be computed by solely studying the additive term in either \eqref{eq.Ch5PsiCont} or \eqref{eq.Ch5PsiCont2}, respectively. We thus only give the results, and refer to \citep{KunzThesis} for the computational details. As before, in the following discussion, the waves' subscripts correspond to the points from which they are obtained. \\

\noindent \textbf{GO components of $\psi$.} Let us assume that $k_2^2 - \a_2^2 >0$. This is always the case if $k_2 > k_1$, but when $k_1 > k_2$, we may have $k_2^2 - \a_2^2 < 0$. If that is the case, we have $\sigma_{sp}' = \emptyset$ and obtain no corresponding  wave components. Indeed, the case of 	$k_2^2 - \a_2^2 < 0$ corresponds to total internal reflection on the wedges exterior. 

We then find
\begin{align}
\psi_{\mathrm{PT}}(\x) = \frac{2 \a_1}{\a_1 - \sqrt{k_2^2 - \a_2^2}} e^{-i(-\sqrt{k_2^2-\a_2^2} x_1 + \a_2 x_2)} \mathcal{H}(x_1)\mathcal{H}(x_2), 
\end{align}
and 
\begin{multline}
\psi_{\mathrm{SR}}(\x) =   
\frac{\a_2 - \sqrt{k_1^2 -k_2^2 + \a_2^2} }
{ \a_2 + \sqrt{k_1^2 - k_2^2 + \a_2^2}} \frac{2\a_1  e^{ i (\sqrt{k_2^2 - \a_2^2}x_1 + \a_2 x_2)} }{\a_1 - \sqrt{k_2^2 - \a_2^2}} \mathcal{H}\left(\a_2 x_1 -  \sqrt{k_2^2 - \a_2^2}  x_2\right)\mathcal{H}(x_2).
\end{multline}
Both, $\psi_{\mathrm{PT}}$ and $\psi_{\mathrm{SR}}$ agree with the corresponding GO components, which are illustrated in Fig. \ref{fig:Wedge04}, right. \\

\noindent \textbf{GO components of $\phi$.} Let us now moreover assume that  $k_1^2 > k_2^2 - \a_2^2 \geq 0$. If $k_1^2 = k_2^2 - \a_2^2$, we get no contribution from the point $\aalpha_{\mathrm{ST}}$ since the crossing becomes tangential. If $k_1^2 < k_2^2 - \a_2^2$, we get no contribution since the singularity $\sigma_{sp}$ does not cross the circle $\sigma_{c_1}$. We will see that as $k_1$ approaches $ \sqrt{k_2^2 - \a_2^2}$, the secondary transmitted wave vanishes. The case $k_1 \leq \sqrt{k_2^2 - \a_2^2}$ thus corresponds to total internal reflection of the primary transmitted wave within the wedges interior. Overall, we then have
{\fontsize{9pt}{0pt} \selectfont	\begin{align}
	\phi_{\mathrm{ST}}(\x) &= 
	\frac{4 \a_1 \a_2 \times e^{ i ( \sqrt{k_2^2 - \a_2^2}x_1 - \sqrt{k_1^2-k_2^2 + \a_2^2} x_2)}}{\left(\a_1 - \sqrt{k_2^2 -\a_2^2}\right)\left(\a_2 + \sqrt{k_1^2 -k_2^2 + \a_2^2}\right)}\mathcal{H}\left(x_1 \sqrt{k_1^2 - k_2^2 + \a_2^2} + x_2 \sqrt{k_2^2 - \a_2^2} \right)\mathcal{H}(-x_2), \label{eq.Ch5PhiST} \\
	\phi_{\mathrm{PR}} (\x) &= 	\frac{\a_1 + \sqrt{k^2_2 - \a^2_2} }{\a_1 - 	\sqrt{k_2^2-\a^2_2}} e^{-i(-\mathfrak{a}_1x_1 + \mathfrak{a}_2 x_2)}\mathcal{H}(-x_1 \a_2  - x_2 \a_1) \mathcal{H}(-x_1), \label{eq.ComplRefl.Primary}
\end{align}		}
and
\begin{align}
\phi_{\mathrm{SH}}(\x) = - e^{-i(\a_1 x_1 + \a_2 x_2)} \mathcal{H}(-x_2)\mathcal{H}(x_1 \a_2 - x_2 \a_1). 
\end{align}
Note that $\phi_{\mathrm{SH}}$ annihilates the incident wave $\phi_{\iin}$ within the region $\{x_1 \a_2 - x_2 \a_1 >0 , x_2<0\}$, which leads to the expected shadow region. The formulae that are given above yield the anticipated GO wave components, which are illustrated in Fig. \ref{fig:Wedge04}, right. 

\begin{remark}
The points $\aalpha_{\mathrm{SH}}$ yields a triple crossing of singularities. To apply the framework developed in \citep{AssierShaninKorolkov2022}, it thus needs to be reduced to a double crossing, which can be done by computing the contributing behaviour of $\Phi_{3/4}$ near $\aalpha_{\mathrm{SH}}$ (see \citep{KunzThesis}). However, in general, not every triple crossing is reducible to a double crossing, and in fact, such crossings do occur in the quarter-plane problem \citep{AssierShaninKorolkov2023}, Section 4.
\end{remark}

We have thus proved the correctness of \eqref{eq.Ch1PWFarfield01} in both, the simple and the complicated case. In the complicated case, the fields' GO components, which are illustrated in Fig. \ref{fig:Wedge04} (right), are given by $\psi_{\GO} = \psi_{\mathrm{PT}} + \psi_{\mathrm{SR}}$ and $\phi_{\GO} = \phi_{\iin} + \phi_{\mathrm{PR}} + \phi_{\mathrm{ST}} + \phi_{\mathrm{SH}}$, and the diffracted far-field is as in the simple case.

\section{Concluding remarks}\label{sec:Ch5Conclusion}

In this article, we have used the machinery developed in \citep{AssierShaninKorolkov2022}  to prove the correctness of equation \eqref{eq.Ch1PWFarfield01},  and thus gave a closed-form description of the far-field encountered in the right-angled no-contrast penetrable wedge diffraction problem. Although the corresponding cylindrical and lateral diffraction coefficients remain unknown, we found that they can be expressed in terms of the spectral functions $\Psi_{++}$ and $\Phi$ as described in Sections \ref{sec:PhiDiffr} (for $\psi_{\mathrm{C}}$), \ref{sec:FirstLateral} (for $\psi_{\mathrm{L}_1}$ and $\psi_{\mathrm{L}_2}$), \ref{sec:PhiDiffr} (for $\phi_{\mathrm{C}}$) and \ref{sec:Ch5LateralPhi} (for $\phi_{\mathrm{L}_1}$ and $\phi_{\mathrm{L}_2}$). Moreover, we have shown that by imposing Sommerfeld's radiation condition via the limiting absorption principle and continuity of bypass on the scattered fields (as opposed to the diffracted fields), we recover the GO components which are predicted by Snell's law and the law of reflection.

We plan on using the system of integral equations presented in Theorem 4.2.1 of \citep{Kunz2021diffraction} to construct rational approximations of $\Psi_{++}$ and $\Phi$, akin to the scheme that was proposed by \cite{AssierAbrahams2020}. It is hoped that such approximations will allow for rapid and accurate computation of the diffraction coefficients. 

The  wave components we have derived are invalid on the GO lines of discontinuity and on the wedges interface. Describing the waves within these regions requires the study of non-isolated saddles on singularities. Developing a framework that allows for this, and obtaining a uniform far-field approximation, will be the focus of future work.

\section*{Acknowledgements}
Valentin D. Kunz would like to acknowledge funding 
by the University of Manchester (Dean's scholarship award).

\bibliographystyle{apalike}
\bibliography{bibliography}

\end{document}